\documentclass[pdflatex,sn-mathphys-num]{sn-jnl}

\usepackage{graphicx}%
\usepackage{multirow}%
\usepackage{amsmath}%
\usepackage{amssymb,amsfonts}
\usepackage{amsthm}%
\usepackage{mathrsfs}%
\usepackage[title]{appendix}%
\usepackage{xcolor}%
\usepackage{textcomp}%
\usepackage{manyfoot}%
\usepackage{booktabs}%
\usepackage{algorithm}%
\usepackage{algorithmicx}%
\usepackage{algpseudocode}%
\usepackage{listings}%
\usepackage{stmaryrd}
\usepackage{tikz}
\usepackage{color}
\usepackage{cleveref}

\usepackage{amsfonts}
\usepackage{mathrsfs}

\usepackage{lineno}

\usepackage{bm}
\usepackage[bbgreekl]{mathbbol}
\DeclareSymbolFontAlphabet{\mathbbm}{bbold}
\DeclareSymbolFontAlphabet{\mathbb}{AMSb}
\usepackage{graphicx}
\usepackage{mathrsfs}
\usepackage{hyperref}

\renewcommand{\phi}{\varphi}

\newcommand{\eps}{\ensuremath{\epsilon}}

\newcommand{\bB}{\ensuremath{\mathcal{B}}}

\newcommand{\fF}{\ensuremath{\mathscr{F}}}

\newcommand{\PP}{\ensuremath{\mathbb{P}}}

\newcommand{\RR}{\ensuremath{\mathbb{R}}}
\newcommand{\NN}{\ensuremath{\mathbb{N}}}

\theoremstyle{thmstyleone}%
\newtheorem{theorem}{Theorem}
\newtheorem{lemma}{Lemma}%
\theoremstyle{thmstyletwo}%
\newtheorem{remark}{Remark}%

\theoremstyle{thmstylethree}%
\newtheorem{definition}{Definition}%

\begin{document}

\title[Almost Sure Averaging for Fast-slow SDEs]{Almost Sure Averaging for Fast-slow Stochastic Differential Equations via Controlled Rough Path}


\author[1,2]{\fnm{Bin} \sur{Pei}}\email{binpei@nwpu.edu.cn}

\author[3]{\fnm{Robert} \sur{Hesse}}\email{robert.hesse@uni-jena.de}

\author[3]{\fnm{Bjeorn} \sur{Schmalfuss}}\email{bjoern.schmalfuss@uni-jena.de}

\author*[1,2]{\fnm{Yong} \sur{Xu}}\email{hsux3@nwpu.edu.cn}

\affil*[1]{\orgdiv{School of Mathematics and Statistics}, \orgname{Northwestern Polytechnical University}, \orgaddress{\street{127 West Youyi Road}, \city{Xi'an}, \postcode{710072}, \country{China}}}

\affil[2]{\orgdiv{MOE Key Laboratory of Complexity Science in Aerospace}, \orgname{Northwestern Polytechnical University}, \orgaddress{\street{127 West Youyi Road}, \city{Xi'an}, \postcode{710072}, \country{China}}}

\affil[3]{\orgdiv{Institut f\"{u}r Stochastik}, \orgname{Friedrich Schiller Universit{\"a}t}, \orgaddress{\street{Ernst Abbe Platz 2}, \city{Jena}, \postcode{D-77043}, \country{Germany}}}

\abstract{This paper applies the averaging method to a coupled system consisting of two stochastic differential equations which has a slow component driven by fractional Brownian motion (FBM) with Hurst-index $1/3< H \leq 1/2$ and a fast dynamics under additive FBM with Hurst-index $1/3< \hat H \leq 1/2$. We prove that the solution of the slow component converges almost surely to the solution of the corresponding averaged equation using the approach of time discretization and controlled rough path. To do this, we employ the random dynamical system (RDS) to obtain a
stationary solution by an exponentially attracting random fixed point of the RDS generated by the non-Markovian fast component.}




\maketitle

\section{Introduction}\label{s1}
The aim of this article is to address the almost sure averaging for fast-slow stochastic differential equations (SDEs) driven by two fractional Brownian motions (FBMs)
\begin{align}
	\label{eq-org-x}
	dX^\eps_t=& f(X^\eps_t,Y^\eps_t)\,dt+h(X^\eps_t)\,dB_t, X^\eps_0=X_0 \in \mathbb{R}^n,  \\
	\label{eq-org-y}
	dY^\eps_t= & \frac{1}{\eps}g(X^\eps_t,Y^\eps_t)\,dt +d\hat B_{t/\eps}, Y^\eps(0)=Y_0 \in \mathbb{R}^m
\end{align}
where $0<\eps \leq 1$ is a small parameter, $X_t^\eps \in \mathbb{R}^n$ and $Y_t^\eps \in \mathbb{R}^m$ are the state variables, $B_t=(B^1_t,\ldots,B^{d_{1}}_t)^{T}$ with Hurst-index $ H\in(1/3, 1/2]$ is a $d_1$-dimensional FBM and $\hat B_t=(\hat B^{1}_t,\ldots,\hat B^{d_{2}}_t)^{T}$ with Hurst-index $\hat H\in(1/3, 1/2]$ is a  $d_2$-dimensional FBM. $B_t$ and $\hat B_t$ are independent FBMs. $f,h,g$ are sufficiently regular.

Rewrite (\ref{eq-org-x})-(\ref{eq-org-y})  in the following form
\begin{align}\label{eq-5}
	\bigg(\begin{array}{c}{dX^\eps_t} \\ {dY^\eps_t}\end{array}\bigg)=\bigg(\begin{array}{c}{f(X^\eps_t,Y^\eps_t)} \\ {\frac{1}{\eps}g(X^\eps_t,Y^\eps_t)}\end{array}\bigg)dt+\bigg(
	\begin{matrix}
		h(X^\eps_t) & O \\
		O & {\rm id}
	\end{matrix}
	\bigg)\bigg(\begin{array}{c}{dB_t} \\ {d\hat B_{t/\eps}}\end{array}\bigg).
\end{align}
A standard scheme to solve above (\ref{eq-5}) is: (i) to give meaning to the integral; (ii) to apply some fixed point result. To deal with item (i) we need an integration theory that is satisfactory in the sense that it allows
to work with signals and unknowns of suitable regularity. To deal with item (ii) we need the space of solutions to (\ref{eq-5}) to have some nice metric
structure. Rough paths theory  allows us to solve SDEs path-wise, not relying on It\^o calculus and its solutions are often defined in the sense of Lyons \cite{lyons1994differential}. It seems that (\ref{eq-5}) can be solved in the path-wise approach by taking a realization of the driving path. It is well-known that the above diffusion term integration in (\ref{eq-5}) does not make sense unless we impose additional structure on $h$,
which is exactly the notion of controlled rough paths as introduced by Gubinelli \cite{gubinelli2004controlling}. 
(\ref{eq-5}) can be understood in the sense of controlled rough paths. This approach to this problem has the advantage of offering a clear break between the
deterministic rough path calculus and the SDEs driven by two FBMs.

It is highly desirable and useful for applications \cite{hasselmann1976stochastic} to find a simplified equation which governs the evolution of the system over the long time scale. The averaging principle for fast-slow SDEs, which was initiated by Khasminskii \cite{khasminskii1968on}, provides a good approximation for the slow component. This research topic seems still quite active, although many papers have been written (see \cite{liu2020averaging,chueshov2005averaging,freidlin2012random} for examples). It should also be recalled that this averaging principle was generalized to various
kind of stochastic systems including jump-type SDEs \cite{xu2011averaging,givon2007strong},
distribution-dependent SDEs \cite{rockner2021strong,hong2022strong,shen2022averaging}, functional-type SDEs \cite{wu2022fast}, fractional driven SDEs  \cite{hairer2019averaging,pei2020averaging,pei2020pathwise,xu2015stochastic,xu2017stochastic}, among others. 
Due to the widely separated time scales and the FBMs in both the slow and fast motions, the fast-slow SDEs driven by two FBMs turn out to be more difficult to deal with than in pure BM case. 
Hairer and Li \cite{hairer2019averaging} considered fast-slow systems
where the slow system is driven by FBM and proved the convergence to the averaged solution took place in probability. Pei et al. \cite{pei2020averaging} answered
affirmatively that an averaging principle still holds for fast-slow mixed SDEs involving both Brownian motion (BM) and FBM $H\in(1/2,1)$ in the mean square sense. Gao et al. consider a qualitatively different approximation problem for rough equations, see \cite{cao2023wong,gao2021rough}.

In all the preceding works mentioned above, the driving noises are either (semi) martingales or have sufficient regularity. One naturally wonders what happens to the averaging principle when both the driving noises are not martingales and have less regularity. Pei et al. \cite{pei2021averaging} investigated the averaging principle for a fast-slow system of rough differential equations (RDEs) driven by mixed fractional Brownian rough paths and proved that when the fast dynamics are Markovian, the slow component driven by FBM with Hurst-index $1/3< H \le 1/2 $ converges to the solution of the corresponding averaged equation in the $L^1$-sense. Later, Inahama \cite{inahama2022averaging,inahama2025averaging} proved the strong averaging principle in the framework of controlled path theory for a fast-slow system of RDEs where the slow and the fast component of the system are driven by a rather general random rough path and Brownian rough path, respectively.

Though it is quite difficult to study the case where the noise in the fast equation is also a FBM. Li and Sieber \cite{li2022slow} made a contribution in that direction by establishing a quantitative quenched ergodic theorem on the conditional evolution of the process of the fast dynamics with frozen slow input. Meanwhile, Pei et al. \cite{pei2023almost} studied the almost sure averaging for a coupled system
consisting of two evolution equations which has a slow component driven by FBM with the Hurst-index $H_1>
1/2$ and a fast component driven by additive FBM with the Hurst-index $ H_2\in(1-H_1,1)$
replacing the invariant measures by the random fixed points of non-Markovian fast motion which are path-wise exponentially attracting.

In the framework of controlled rough paths, the present paper applies methods from \cite{pei2023almost} to deal with finite-dimensional system  (\ref{eq-org-x})-(\ref{eq-org-y}) where both the slow  and fast components are driven by FBM with regularity $1/3< H , \hat{H}\leq 1/2$.
There are several major technical difficulties one encounters, when trying to establish an averaging procedure which proves that the solution of the slow component converges almost surely to the solution of the corresponding averaged equation. A first
conceptual difficulty is to employ the random dynamical system (RDS) to obtain a
stationary solution by an exponentially attracting random fixed point of the RDS generated by the non-Markovian fast component. The second main, more technical, obstacle we encounter is due to the fact that we want to include two FBMs with regularity $1/3< H,\hat H \leq 1/2$. Rough path techniques provide a very natural framework to
obtain the unique solution and some nice norm estimates.

The paper is organized as follows. In Section \ref{s2} we
review some basic concepts of rough paths, rough integrals, RDE and FBM that are used in the paper. Section \ref{s3} contains the existence and uniqueness of a path-wise solution and almost sure averaging to the fast-slow SDEs driven by two FBMs. The random fixed points
for the RDS generated by (\ref{eq-org-y}) are considered in Section \ref{s3-2}. An averaging principle which proves that the solution of the slow component converges almost surely to the solution of the corresponding averaged equation was established in Section \ref{s4}. Note that a constant $C$ appears very often and it may change form line to line.

\section{Preliminaries}\label{s2}
In this section we review some basic concepts of rough paths, rough integrals, RDE and FBM that will be used later.
\subsection{Rough paths}
In this subsection, we will recall some facts about rough paths.  For a compact time interval  $I=[T_{1}, T_{2}] \subset \mathbb{R}$ , we write  $|I|=T_{2}-T_{1}$  and  $I^{2}=\{(s, t) \in I \times I: s \leq t\}$.
Let $V$ and $W$ be finite-dimensional Euclidean spaces, so that the continuous linear maps from $V$ to $W$ form a Banach space $\mathcal{L}(V,W)$. Although the discussion could be set for general Banach spaces, we restrict here to finite-dimensional Euclidean $V,W$. We then identify
$
\mathbb{R}^{m\times n}\cong \mathcal{L}(\mathbb{R}^n,\mathbb{R}^m),
$
with $\mathbb{R}^n$ denoting Euclidean space, whose norm and inner product are denoted by $||\cdot||$ and $(\cdot,\cdot)$, respectively.

We denote by $\mathcal{C}(I ; V)$  the space of all continuous functions  $y: I \rightarrow V $ equipped with the norm  $\|\cdot\|_{\infty, I}$  given by  $\|y\|_{\infty, I}=\sup _{t \in I}\|y_{t}\|$, and let  $\mathcal{C}^1(I, V)$  denote the space of the first order continuously differentiable $V$-valued functions on  $I$.

We write  $y_{s, t}=y_{t}-y_{s}$. For $p \geq 1$, we denote by $ \mathcal{C}^{p-{\rm var}}(I ; V)$  the set of all continuous functions $ y: I \rightarrow V$  which have a finite $p$-variation
$$\interleave y \interleave_{p-{\rm var}, I}=\Big(\sup _{\mathcal{P}(I)} \sum_{[t_{i}, t_{i+1}] \in \mathcal{P}(I)}\|y_{t_{i+1}}-y_{t_{i}}\|^{p}\Big)^{\frac{1}{p}}<\infty$$
where  $\mathcal{P}(I)$  is a partition of the interval $I$.
Furthermore, we equip this space with the norm
$$\|y\|_{p-{\rm var}, I}:=\|y_{T_1}\|+\interleave y \interleave_{p-{\rm var}, I} .$$
This norm is equivalent to
$$\|y\|_{p-{\rm var}, I}:=\|y\|_{\infty, I}+\interleave y \interleave_{p-{\rm var}, I} .$$
For properties of the  $p$-variation norm we refer to \cite{lyons1994differential}. Also for each  $0<\alpha<1$ , we denote by  $\mathcal{C}^{\alpha}(I, V) $ the space of H\"older continuous functions with exponent  $\alpha $ on  $I$  equipped with the norm
$$\|y\|_{\alpha, I}:=\left\|y_{T_1}\right\|+\interleave y \interleave_{\alpha, I}$$
or the equivalent norm
$$\|y\|_{\alpha, I}:=\|y\|_{\infty, I}+\interleave y \interleave_{\alpha, I}$$
where $\interleave y \interleave_{\alpha, I}:=\sup _{s, t \in I, s<t} \frac{\left\|y_{s, t}\right\|}{(t-s)^{\alpha}}<\infty $.

Similarly for functions of two variables $R_{\cdot,\cdot}:I^2\rightarrow V$, we define
\begin{equation*}
  \interleave R \interleave_{p-{\rm var}, I^2}=\Big(\sup _{\mathcal{P}(I)} \sum_{[t_{i}, t_{i+1}] \in \mathcal{P}(I)}\|R_{t_{i}, t_{i+1}}\|^{p}\Big)^{\frac{1}{p}},  
\end{equation*}
and 
\begin{equation*}
   \interleave R \interleave_{\alpha, I^2}:=\sup _{(s, t) \in I^2} \frac{\left\|R_{s, t}\right\|}{(t-s)^{\alpha}}<\infty. 
\end{equation*}

\begin{definition}{\rm \cite[p.
		14]{friz2020course}}
	For  $\alpha \in\left(\frac{1}{3}, \frac{1}{2}\right]$, a pair  $\mathbf{x}=(\mathrm{x}, \mathbb{X}) \in \mathcal{C}^{\alpha}(I, \mathbb R^d)\times \mathcal{C}^{2 \alpha}(I^{2}, \mathbb R^d \otimes \mathbb R^d) $ is called rough path if it satisfies the analytic properties
	$$
	\interleave \mathrm{x} \interleave_{\alpha}=\sup _{s<t \in I} \frac{\|\mathrm{x}_{s, t}\|}{|t-s|^{\alpha}}<\infty, \quad \interleave \mathbb{X} \interleave_{2 \alpha}=\sup _{s<t \in I} \frac{\|\mathbb{X}_{s, t}\|}{|t-s|^{2 \alpha}}<\infty$$
	and Chen's relation
	$$
	\mathbb{X}_{s, t}=\mathbb{X}_{s, u}+\mathbb{X}_{u, t}+\mathrm{x}_{s, u} \otimes \mathrm{x}_{u, t}$$
	for  s $\leq u \leq t \in I $, we denote by the Banach space $\mathscr{C}^{\alpha}(I, \mathbb R^d)$  the space of rough paths equipped with the homogeneous rough path norm {\rm \cite[p.18]{friz2020course}}
	\begin{eqnarray*}
		\interleave \mathbf{x} \interleave_{\alpha, I}=\interleave \mathrm{x} \interleave_{\alpha, I}+\interleave \mathbb{X} \interleave_{2 \alpha, I^{2}}^{\frac{1}{2}}.
	\end{eqnarray*}
\end{definition}	
	In addition, for any  $\mathrm{x} \in \mathcal{C}^1(I,\mathbb R^d)$, there is a canonical lift  $S(\mathrm{x}):=(\mathrm{x}, \mathbb{X})$  in  $\mathscr{C}^{\alpha}(I, \mathbb R^d) $ defined as
	\begin{eqnarray}
		\mathbb{X}_{s, t}^{k, l}=\int_{s}^{t} \int_{s}^{r} d \mathrm{x}_{r^{\prime}}^{k} d \mathrm{x}_{r}^{l}=\int_{s}^{t} \mathrm{x}_{s,r}^{k} d \mathrm{x}_{r}^{l}, \quad s<t \in I \quad {\rm and} \quad
		k, l \in\{1, \cdots, d\} .
	\end{eqnarray}
	We denote the space $\mathscr{C}_{g}^{0,\alpha}(I, \mathbb R^d) $  as the geometric rough space, i.e. the closure of the canonical lift  $S(\mathrm{x})$ in $\mathscr{C}^{\alpha}(I,\mathbb R^d) $.

The first component  $\mathrm{x} $  is the path component and the second component  $\mathbb{X}$  is called the second order process. In addition, for $p=\frac{1}{\alpha}$, we have the  $p$-variation norm
\begin{eqnarray}\label{xxx}
	\interleave \mathbf{x} \interleave_{p-{\rm var}, I}=(\interleave \mathrm{x} \interleave_{p-{\rm var}, I}^{p}+\interleave \mathbb{X} \interleave_{\frac p2-{\rm var}, I^{2}}^{\frac p2})^{\frac{1}{p}}<\infty
\end{eqnarray}
Let  $\mathscr{C}^{p-{\rm var}}(I, \mathbb R^d)$  denote the space of all rough paths which have a finite  $p$-variation norm. It is clear that $\mathscr{C}^{\alpha}(I, \mathbb R^d) \subset \mathscr{C}^{p-{\rm var}}(I,\mathbb R^d)$ .

\begin{lemma}\label{var}{\rm (cf. \cite[Exercise 5.11]{friz2010multidimensional} and \cite[Lemma 2.1]{cong2018nonautonomous})}
	Let $\mathbf{x} \in \mathcal{C}^{p-{\rm var}}([s,t];\mathbb R^d)$, \, $p > 1$. For any partition $\mathcal{P}(s,t)$ of the interval $[s,t]$ given by $s=u_1<u_2< \cdots < u_n=t$, we have
	$$\sum_{i=1}^{n-1} \interleave \mathbf{x} \interleave_{p-{\rm var}, [u_i,u_{i+1}]}^{p} \leq \interleave \mathbf{x} \interleave_{p-{\rm var}, [s,t]}^{p} \leq (n-1)^{p-1} \sum_{i=1}^{n-1} \interleave \mathbf{x} \interleave_{p-{\rm var}, [u_i,u_{i+1}]}^{p}.$$
\end{lemma}
\begin{definition}{\rm (cf. \cite[p.
		84]{friz2020course} and \cite[Lemma 2.3]{cong2018nonautonomous})}
A control on $I$ is a continuous map $\ell:(s,t)\in I\times I\to[0,\infty)$ that is super-additive, in the sense that it satisfies the following conditions:
	\begin{enumerate}
		\item[(i)] for all  $t \in I$, $\ell_{t, t}=0,$
		\item[(ii)] for all $s \leq t \leq u  \in  I,$ $\ell_{s, t}+\ell_{t, u} \leq \ell_{s, u}$.
	\end{enumerate}
\end{definition}
Typical examples of controls include 
$(s,t)\longmapsto (t-s)^\theta,\theta\ge 1,$
and
$(s, t) \longrightarrow \interleave \mathrm{x} \interleave^{p}_{p-{\rm var}, [s,t]}, p\geq 1$, 
where $\mathrm{x}$ is a path of finite $p$-variation on $I$. More generally, if $\mathrm{x}$ and $\bar{\mathrm{x}}$ are paths of finite $p$-variation on $I$, then
\begin{align}\label{controlex}
	(s, t) \longrightarrow \interleave \mathrm{x} \interleave^{\beta p}_{p-{\rm var}, [s,t]} \cdot \interleave \bar{\mathrm{x}} \interleave^{\alpha p}_{p-{\rm var}, [s,t]},
\end{align}
is also a control, provided that $\alpha,\beta\ge 0$ and $\alpha+\beta\ge 1$. These examples are standard; see \cite[Proposition 5.8, p.80; Exercise 1.9, p.22]{friz2010multidimensional}.

The following lemma gives a useful property of controls in order to estimate the variational norm of  paths and functions of two variables, respectively.

\begin{lemma}\label{control}
	Let  $(\ell_j)_{j\in\mathbb{N}^{+}}$  be control functions on  $I$  and $C>0, p,q\geq 1, n \in \mathbb{N}^{+}$. 
	\begin{enumerate}
		\item[(i)] For a continuous path $\mathrm{x}:I \rightarrow \mathbb R^d$  satisfying
		$$\|\mathrm{x}_{s,t}\| \leq  C \sum_{j=1}^{n}\ell_{j,s,t}^{1 / p}, \quad \forall s<t \in I$$
		one has 
		$$\interleave \mathrm{x} \interleave_{p-{\rm var}, [s,t]} \leq  C n^{\frac {p-1} p} \sum_{j=1}^{n}\ell_{j,s,t}^{1 / p}, \quad \forall s<t \in I 	.$$
		\item[(ii)]
		For continuous functions of two variables $R_{\cdot,\cdot}:I^2\rightarrow \mathbb R^d$,
		satisfying $$\|R_{s,t}\| \leq  C \sum_{j=1}^{n}\ell_{j,s,t}^{1 / q}, \quad \forall s<t \in I, \quad $$
		we have
		$$\interleave R \interleave_{q-{\rm var}, [s,t]^2} \leq C n^{\frac {q-1} q} \sum_{j=1}^{n}\ell_{j,s,t}^{1 / q}, \quad \forall s<t \in I. $$
	\end{enumerate}
\end{lemma}
\begin{proof}
	Similar to {\rm \cite[Proposition 5.10, p.83]{friz2010multidimensional}}, we can prove as follows
	\begin{align*}
		\begin{split}
			\interleave \mathrm{x} \interleave_{p-{\rm var}, [s,t]}&=\Big(\sup _{\mathcal{P}([s,t])} \sum_{[t_{i}, t_{i+1}] \in \mathcal{P}([s,t])}\|x_{t_{i}, t_{i+1}}\|^{p}\Big)^{\frac{1}{p}}\\
			& \leq \Big(\sup _{\mathcal{P}([s,t])} \sum_{[t_{i}, t_{i+1}] \in \mathcal{P}([s,t])} C^p \big(\sum_{j=1}^{n}\ell^{1/p}_{j,t_{i}, t_{i+1}}\big)^p \Big)^{\frac{1}{p}}\\
			& \leq \Big(\sup _{\mathcal{P}([s,t])} \sum_{[t_{i}, t_{i+1}] \in \mathcal{P}([s,t])} C^p n^{p-1} \sum_{j=1}^{n}\ell_{j,t_{i}, t_{i+1}} \Big)^{\frac{1}{p}}\\
			& \leq  C n^{\frac {p-1} p} \big(\sum_{j=1}^{n}\ell_{j,s,t}\big)^{1 / p} \leq  C n^{\frac {p-1} p} \sum_{j=1}^{n}\ell_{j,s,t}^{1 / p}
		\end{split}
	\end{align*}
	where we use the elementary inequalities, $a_j \geq 0, j=1, 2, \cdots, n$
	\begin{align*}
		\begin{split}
			\big(\sum_{j=1}^n a_j\big)^p \leq n^{p-1} \sum_{j=1}^n a^p_j, \quad 
			\big(\sum_{j=1}^n a_j\big)^{1/p}\leq \sum_{j=1}^n a_j ^{1/p}
		\end{split}
	\end{align*}
	for some constant $p\geq 1, n\in \mathbb{N}^{+}$.
	We can prove (ii) in a similar way and thus omit it. 
\end{proof}

\subsection{Rough integrals}\label{2.2}
We now define the rough integral in terms of control functions; for further details, see the monograph by Friz and Hairer \cite{friz2020course}.

Following \cite[Definition 4.6, p.56]{gubinelli2004controlling}, a path  $y \in \mathcal{C}^{\alpha}(I, \mathbb R^n)$  is called to be controlled by  $\mathrm{x} \in   \mathcal{C}^{\alpha}(I, \mathbb R^d)$  if there exists $y^{\prime} \in \mathcal{C}^{\alpha}(I, \mathcal{L}(\mathbb R^d, \mathbb R^n))$, such that the remainder term $R^{y} \in \mathcal{C}^{2 \alpha}(I^{2},\mathbb R^n )$ given implicitly through the relation
\begin{eqnarray}\label{yr}
	R_{s, t}^{y}=y_{s, t}-y_{s}^{\prime} \mathrm{x}_{s, t}, \quad \forall \, T_1 \leq s \leq t \leq T_2.
\end{eqnarray}
$y^{\prime}$  is called Gubinelli-derivative of  $y$, which is uniquely defined as long as $x$ is truly rough (see \cite[Proposition 6.4]{friz2020course}). Denote by $\mathcal{D}_{\mathrm{x}}^{2 \alpha}(I,\mathbb R^n)$  the space of all couples  $(y, y^{\prime})$ that is controlled by $x$, then  $\mathcal{D}_{\mathrm{x}}^{2 \alpha}(I,\mathbb R^n) $ is a Banach space equipped with the norm
$$\|(y, y^{\prime})\|_{\mathrm{x}, 2 \alpha, I} :=\|y_{T_1}\|+\|y_{T_1}^{\prime}\|+\interleave (y, y^{\prime}) \interleave_{\mathrm{x}, 2 \alpha, I}$$ where
$$
\interleave(y, y^{\prime})\interleave_{\mathrm{x}, 2 \alpha, I} :=\interleave y^{\prime}\interleave_{\alpha, I}+\interleave R^{y}\interleave_{2 \alpha, I^{2}}.$$

For a fixed rough path  $\mathbf{x}=(\mathrm{x}, \mathbb{X})$  and any controlled rough path  $(y, y^{\prime}) \in   \mathcal{D}_{\mathrm{x}}^{2 \alpha}(I,\mathcal{L}(\mathbb R^d, \mathbb R^n)),$ it is proved in  \cite[Proposition 1]{gubinelli2004controlling} using the sewing lemma that the rough integral can be defined as
\begin{eqnarray}\label{2.4}
	\int_{s}^{t} y_{u} d \mathbf{x}_{u}:=\lim _{|\Pi| \rightarrow 0} \sum_{[u, v] \in \Pi}\left(y_{u}  \mathrm{x}_{u, v}+y_{u}^{\prime} \mathbb{X}_{u, v}\right)
\end{eqnarray}
where we used the canonical injection  $\mathcal{L}(\mathbb R^d,\mathcal{L}(\mathbb R^d, \mathbb R^n))\hookrightarrow \mathcal{L}(\mathbb R^d\otimes \mathbb R^d, \mathbb R^n) $ in writing $y_{u}^{\prime} \mathbb{X}_{u, v}$ and the limit is taken on all the finite partition  $\Pi$  of  $I$  with  $|\Pi|:=   \max _{[u, v] \in \Pi}|v-u|$. With these notions, the resulting integral (\ref{2.4}) takes values in $\mathbb R^n$. Moreover, there exists a constant  $C_{\alpha}=C_{\alpha,|I|}>1$, such that
\begin{align}\label{eq2.5}
	\begin{split}
		\bigg\|\int_{s}^{t} y_{u} d \mathbf{x}_{u}&-y_{s} \mathrm{x}_{s, t}-y_{s}^{\prime} \mathbb{X}_{s, t}\bigg\| \\
		&\leq C_{\alpha}|t-s|^{3 \alpha} (\|\mathrm{x}\|_{\alpha,[s, t]}\interleave R^{y} \interleave_{2 \alpha,[s, t]^{2}}+\|y^{\prime}\|_{\alpha,[s, t]}\|\mathbb{X}\|_{2 \alpha,[s, t]^{2}}) .
	\end{split}
\end{align}
Throughout the paper, we will occasionally write $\|\mathrm{x}\|_{\alpha}$ or  $\|\mathbb{X}\|_{2 \alpha}$ without explicitly indicating the underlying domains $I $ or $ I^{2}$, respectively. In practice, we will also make use of the $p$-variation norm
\begin{align}\label{2.7}
	\|(y, y^{\prime})\|_{\mathrm{x}, p} :=\|y_{T_1}\|+\|y_{T_1}^{\prime}\|+\interleave (y, y^{\prime})\interleave_{\mathrm{x}, p}
\end{align}
where
\begin{align*}
	\interleave (y, y^{\prime}) \interleave_{\mathrm{x}, p} :=\interleave y^{\prime} \interleave_{p-{\rm var}}+\interleave R^{y}\interleave_{q-{\rm var}}.
\end{align*}

Thanks to the sewing lemma \cite{friz2020course}, we can use a similar version to (\ref{eq2.5}) under  $p$-variation norm as follows.
\begin{align}\label{inq-1}
	\begin{split}
		\bigg\|\int_{s}^{t} y_{u}&  d \mathbf{x}_{u}-y_{s}  \mathrm{x}_{s, t}-y_{s}^{\prime} \mathbb{X}_{s, t}\bigg\| \\
		&\leq C_{p}(\interleave \mathrm{x} \interleave_{p-{\rm var},[s, t]}\interleave R^{y}\interleave_{q-{\rm var},[s, t]^{2}}+\interleave y^{\prime}\interleave_{p-{\rm var},[s, t]}\interleave \mathbb{X} \interleave_{q-{\rm var},[s, t]^{2}})
	\end{split}
\end{align}
with constant  $C_{p}>1$  independent of  $\mathbf{x}$  and  $y$ .
\subsection{Rough Differential Equation}\label{2.3}
Consider the RDE
\begin{align}\label{eq-1}
	d y_{t}=F(y_{t}) d t+G (y_{t}) d \mathbf{x}_{t}, \quad \forall t \in I, \quad y_{T_1}\in\mathbb R^n
\end{align}
where $\mathbf{x}$ is a rough path. Such system is understood as a path-wise approach to solve a SDE driven by a H\"older continuous stochastic process which can be lifted to a rough path.

To study the RDE (\ref{eq-1}), we impose the following assumptions.
\begin{enumerate}
	\item[(H1)]  $F: \mathbb R^n \rightarrow \mathbb R^n$  is globally Lipschitz continuous with the Lipschitz constant  $C_{F}$;
	\item[(H2)] $G$  belongs to $C_{b}^{3}(\mathbb R^n, \mathcal{L}(\mathbb R^d,\mathbb R^n)) $ such that
	$$C_{G}:=\max \{\|G\|_{\infty},\|DG\|_{\infty},\|D^{2}{G}\|_{\infty},\|D^{3}{G}\|_{\infty}\}<\infty .$$
\end{enumerate}

\begin{lemma}\label{solution} (\cite[Theorem 3.1]{duc2020controlled})
Let $\alpha \in(1/3,1/2)$ and let $\mathbf{x}=(\mathrm{x}, \mathbb X) \in\mathcal{C}^{\alpha}(I,\mathbb{R}^d)$ be a rough path. Suppose that assumptions $(H1)$-$(H2)$ hold. Then there exists a unique controlled path $(y, y^\prime) \in \mathcal{D}_{\mathrm{x}}^{2 \alpha}(I,\mathbb R^n)$ such that $y^\prime=G(y)$, and such that
 \begin{align}\label{eq-2}
	y_{t}=y_{T_1}+\int_{T_1}^{t} F(y_{s}) d s+\int_{T_1}^{t} G(y_{s}) d \mathbf{x}_{s}, \quad \forall  t \in I
\end{align}
for any initial value  $y_{T_1} \in \mathbb R^n$.   
\end{lemma}

Riedel et al. \cite{riedel2017rough} studied controlled differential equations driven by a rough path (in the sense of T. Lyons) with an
additional, possibly unbounded drift term while $G$ to be bounded and sufficiently smooth and showed that the equation induces a solution flow if the drift
grows at most linearly. Searching for a solution in the Gubinelli sense $(y, y^\prime) \in \mathcal{D}_{\mathrm{x}}^{2 \alpha}(I,\mathbb R^n)$, is possible because for $G: \mathbb R^n \rightarrow  \mathcal{L}(\mathbb R^n,\mathbb R^n)$ satisfying  (H2), one has
\begin{align*}
	(y, y^{\prime}) \in \mathcal{D}_{\mathrm{x}}^{2 \alpha} (I,\mathbb R^n)\Rightarrow &(G(y),D G(y) y^{\prime}) \in \mathcal{D}_{\mathrm{x}}^{2 \alpha}(I,\mathcal{L}(\mathbb R^d,\mathbb R^n)).
\end{align*}
The proof of \cite[Theorem 3.1]{duc2020controlled} relies on the Doss–Sussmann technique  \cite{sussmann1978gap} and the sequence of stopping times introduced in \cite{cass2013integrability}. Specifically, for any fixed $\nu\in(0,1)$, define the sequence of stopping times $\{\tau_i(\nu,\mathbf{x},I)\}_{i\in\mathbb N}$ recursively by
$$
\tau_{0}=T_1, \quad \tau_{i+1}:=\inf \left\{t>\tau_{i}:\interleave\mathbf{x}\interleave_{p-{\rm var},\left[\tau_{i}, t\right]}=\nu\right\} \wedge T_2 .
$$

Setting
$N_{\nu,I,p}(\mathbf{x}) :=\sup \left\{i \in \mathbb{N}: \tau_{i} \leq T_2\right\}$,
we obtain the following rough estimate:
$$
N_{\nu,I,p}(\mathbf{x}) \leq 1+\nu^{-p}\interleave\mathbf{x}\interleave_{p-\mathrm{var}, I}^{p} .
$$

Other studies on properties of stopping times can also be founded in  \cite[ Section 4]{duc2018exponential}.

The next result provides the norm estimates for the solution of (\ref{eq-1}).
\begin{lemma}
The supremum and  $p$-variation norms of the solution to (\ref{eq-1}) are estimated as follows
	\begin{align*}
		\|y\|_{\infty,[T_1, T_2]} \leq &\big(\|y_{T_1}\|+({\|F(0)\|}{C_{F}^{-1}}+{C_{p}^{-1}}) N_{\frac{1}{4C_p C_G},[T_1, T_2],p}(\mathbf{x})\big)e^{4 C_{F}(T_2-T_1)}, \\
		\interleave y, R^{y}\interleave_{p-{\rm var},[T_1, T_2]}
		\leq& \big(\|y_{T_1}\|+({\|F(0)\|}{C_{F}^{-1}}+{C_{p}^{-1}})  N_{\frac{1}{4C_p C_G},[T_1, T_2],p}(\mathbf{x})\big)\\
		&\quad \times  \exp{\{4 C_{F}(T_2-T_1)\}} N_{\frac{1}{4C_p C_G},[T_1, T_2],p}^{\frac{p-1}{p}}(\mathbf{x})-\|y_{T_1}\|
	\end{align*}
	where $\interleave y, R^{y}\interleave_{p-{\rm var},[T_1, T_2]}=\interleave y \interleave_{p-{\rm var},[T_1, T_2]}+\interleave R^{y}\interleave_{q-{\rm var},[T_1, T_2]^2}$.
\end{lemma}

\subsection{Fractional Brownian Motion}
Let $B$ be an $\mathbb{R}^{d_1}$-valued FBM on $[0,T]$ defined on some probability spaces and let $\hat{B}$ be another FBM on $\mathbb{R}$ with values in $\mathbb{R}^{d_2}$  independent of $B$. $B$ has the Hurst-index $H\in (\frac13, \frac12]$ and $\hat B$ has the same Hurst-index $\hat H\in (\frac13, \frac12]$.  In particular, $B$ has covariance function
\begin{align}\label{cova}
	R_{B}(s,t)=\frac12 (|t|^{2H}+|s|^{2H}-|t-s|^{2H})Id,\quad
	t,\,s\in [0,T]
\end{align}
where $Id$ is the identity matrix. The covariance of $\hat B$, i.e. $R_{\hat B}(s,t)$, can be defined by (\ref{cova}) replacing $ H$ by $\hat H$ and $[0,T]$ by $\mathbb{R}$.

By Bauer \cite[Theorem 38.6]{MR1385460} applied to the FBM we have the canonical versions
$B(\omega_{1})=\omega_1$, $\hat B(\omega_2)=\omega_2$,  for the probability spaces
$$(\mathcal{C}_0([0,T],\RR^{d_1}),\bB(\mathcal{C}_0([0,T],\RR^{d_1})),\PP_{H})\,\,{\rm and}\,\,(\mathcal{C}_0(\RR,\RR^{d_2}),\bB(\mathcal{C}_0(\RR,\RR^{d_2})),\PP_{\hat H})$$
where $\mathcal{C}_0(\RR,\RR^{d_2})$ is the (metrizable) space of continuous functions
on $\RR$ with values in $\RR^{d_2}$ and with value zero at zero equipped
with the compact open topology. $\bB(\mathcal{C}_0(\RR,\RR^{d_2}))$ is the Borel $\sigma$-algebra of $\mathcal{C}_0(\RR,\RR^{d_2})$. $\PP_{\hat H}$ is the Gaussian distribution of the FBM $\hat B$. $\mathcal{C}_0([0,T],\RR^{d_1}),\bB(\mathcal{C}_0([0,T],\RR^{d_1}))$ and $\PP_{H}$ can be defined in a similar way.

Considering $B(\omega_1)=\omega_1$ we have a  set in $\mathcal{C}_{0}([0,T],\mathbb{R}^{d_1})$ of full measure so that $\omega_1$ is $\beta$-H\"older continuous with $\beta\in(\frac13,H)$. In addition there exists a set of full measure $\mathbb{P}_{H}$ so that the rough path $\boldsymbol{\omega_1}=(\omega_1,\bbomega_1)  \in \mathscr{C}_{g}^{0,\beta}([0, T], \mathbb{R}^{d_1})$ in the framework of Section \ref{2.3} is well defined which follows by \cite[Corollary 10.10]{friz2020course}.
In particular
\begin{align*}
	\bbomega^{i,j}_{1,s,t}=&\int_{s}^{t}\omega_{1,s,r}^{i}d \omega_1^j(r), \quad i\ne j\\
	\bbomega^{i,i}_{1,s,t}=&\frac{(\omega_{1,s,t}^{i})^2}{2},  \quad i = j .\end{align*}
Note that $\omega_{1}$ has independent components.

Let us denote the intersection of the both sets of full measure from above by $\Omega_1$. Then we consider the restricted probability space with trace $\sigma$-algebra
\[(\Omega_1,\bB(\mathcal{C}_0(\mathbb R;\mathbb{R}^{d_1}))\cap
\Omega_1, \PP_{H}(\cdot\cap\Omega_1))=(\Omega_1,\mathscr{F}_1,\mathbb{P}_{H}).\]
For simplicity, we write $\PP_{H}(\cdot\cap\Omega_1)$ as $\PP_{H}(\cdot)$ in the following and $\mathscr{F}_1$ denotes the trace  $\sigma$-algebra.

Let $\eta\in (1/3,\hat H)$. Similarly for $\omega_2$, let $\Omega_2$ be the set of $\omega_2$
so that $\omega_2$ is $\eta$-H\"older continuous over any interval $[T_1,T_2], T_1<T_2 \in \mathbb{R}$ and $\|\omega_{2}(t)\|$
has a sub-linear growth
for $t\rightarrow \pm \infty:$
\[\frac{\|\omega_{2}(t)\|}{t}\rightarrow 0, \quad t\rightarrow \infty.\]

This set has measure one, see Cheridito et al. \cite[Proposition A1]{cheridito2003fractional}.
Then, we have the restricted probability space
\[(\Omega_2,\bB(\mathcal{C}_0(\mathbb R;\mathbb{R}^{d_2}))\cap
\Omega_2,\PP_{\hat H}(\cdot\cap\Omega_2))=(\Omega_2,\mathscr{F}_2,\mathbb{P}_{\hat H})\]
where we will write
$\PP_{\hat H}(\cdot)=\PP_{\hat H}(\cdot\cap\Omega_2)$ and $\mathscr{F}_2$ is the trace $\sigma$-algebra.

Let $\theta$ be the measurable and
measure preserving flow on $$(\mathcal{C}_0(\RR,\RR^{m}),\bB(\mathcal{C}_0(\RR,\RR^{m})),\PP_{\hat H})$$ defined by the Wiener-shift
\[
\theta_t\omega_2(\cdot)=\omega_2(\cdot+t)-\omega_2(t), t\in \RR,\,\omega_2\in \mathcal{C}_0(\RR,\RR^{m}).
\]
In particular we have
\[
\theta_t\theta_s\omega_2=\theta_{t+s}\omega_2,\quad s,\, t\in\RR,\quad \theta_0={\rm
	id}_{\mathcal{C}_0(\RR,\RR^{m})}.
\]

We now introduce a subset of full measure in $\mathcal{C}_0(\RR,\RR^{m})$ which is $(\theta_t)_{t\in\RR}$-invariant.
Note that $\PP_{\hat H}$ is ergodic w.r.t. $\theta$. $\Omega_2$ is $(\theta_t)_{t\in\mathbb{R}}$-invariant.

Hence the $\mathbb{P}_{\hat H}$ from the restricted probability space is ergodic w.r.t. $\theta$ restricted to $\Omega_2$. We define $(\Omega,\mathscr{F},\mathbb{P})=(\Omega_1 \times \Omega_2 ,\mathscr{F}_1\otimes\mathscr{F}_2,\mathbb{P}_{H}\times\mathbb{P}_{\hat H}))$ to be the product probability space of the both probability spaces from above.

Later we replace $\omega_2$ by $\omega_{2,\eps}=\omega_2(\eps^{-1}\cdot),\,\eps \in (0,1]$, which is a (non-standard) FBM with covariance
\[
\frac{1}{2\eps^{2\hat H}}(|t|^{2\hat H}+|s|^{2\hat H}-|t-s|^{2\hat H})Id,\quad
t,\,s\in \RR.
\]
Note that $\omega_2$ is $\eta$-H\"older continuous and has a sub-linear growth if and only if this holds for $\omega_{2,\eps}$. Hence $\Omega_2$ can be chosen independently of $\eps$.

Then by Theorem \cite[Theorem 10.4]{friz2020course} we can describe the
H\"older
continuous and canonical  version of $\omega$
by the
probability space
$(\Omega,\fF,\PP)$ with paths
$\omega \in \Omega$ and consider the rough path $\boldsymbol{\omega}=(\omega,\bbomega)  \in \mathscr{C}_{g}^{0,\gamma}([0, T], \mathbb{R}^{d_1+d_2})$ under the framework of rough path theory which was introduced in Sections \ref{2.2} and \ref{2.3}.

\section{System of fast-slow Stochastic Differential Equations}\label{s3}
\subsection{Existence, uniqueness, and a priori estimates} In this subsection,
we are interested in solving the system
\begin{align}\label{slowpath}
	dX^\eps_t&= f(X^\eps_t,Y^\eps_t)\,dt+h(X^\eps_t)\,d\omega_1(t), \quad X^\eps_0=X_0\in \mathbb{R}^n, \\
	\label{fastpath}
	dY^\eps_t&=\frac{1}{\eps}g(X^\eps_t,Y^\eps_t)\,dt +d\omega_{2,\eps}(t), \quad Y^\eps(0)=Y_0 \in \mathbb{R}^m
\end{align}
where $\omega_1$ and $\omega_2$ are sample paths of the canonical FBMs with Hurst parameters $H$ and $\hat H$, respectively, as introduced in Section $\ref{s2}$, and $\omega_{2,\eps}(\cdot)=\omega_2(\frac{1}{\eps}\cdot)$.
We now consider the above equations within the framework of rough path theory, under which equations (\ref{slowpath})-(\ref{fastpath}) satisfy  
\begin{align}\label{slowpathint}
	X^\eps_t=& X_0+\int_{0}^{t}f( X^\eps_r,Y^\eps_r)\,dr+\int_{0}^{t}h(X^\eps_r)\,d\boldsymbol{\omega_1}(r)\\
	\label{fastpathint}
	Y^\eps_t=& Y_0+\frac{1}{\eps}\int_{0}^{t} g( X^\eps_r,Y^\eps_r)\,dr+\int_{0}^{t} \,d\boldsymbol{\omega_{2,\eps}}(r).
\end{align}
where $\boldsymbol{\omega_1}=(\omega_1,\bbomega_1)  \in \mathscr{C}_{g}^{0,\beta}([0, T], \mathbb{R}^{d_1})$ and $\boldsymbol{\omega_{2,\eps}}=(\omega_{2,\eps},\bbomega_{2,\eps})  \in \mathscr{C}_{g}^{0,\eta}([0,T], \mathbb{R}^{d_2})$.

To investigate the almost sure averaging of system (\ref{eq-org-x})-(\ref{eq-org-y}), it is essential to obtain unique solution. Thus,  taking $\eps=1$ without loss of generality, for $y_t=(X^\eps_t,Y^\eps_t)\in \mathbb{R}^{n}\times \mathbb{R}^{m}$, we understand (\ref{slowpath})-(\ref{fastpath}) as follows
\begin{eqnarray}\label{3.5}
	d y_t=F (y_t)dt+G(y_t)(d\boldsymbol{\omega}_{1}(t), d \boldsymbol{\omega}_{2}(t))
\end{eqnarray}
where $y_0=(X_0,Y_0)$, $F:=(f,g): \mathbb{R}^{n}\times \mathbb{R}^{m} \rightarrow \mathbb{R}^{n}\times \mathbb{R}^{m}$, $G$ is a an operator such that
\begin{align} G(y_t):=\bigg(
	\begin{matrix}
		h(X^\eps_t) & 0 \\
		0 & {\rm id}
	\end{matrix}
	\bigg).
\end{align}

We interpret the equation (\ref{3.5}) in the form of an integral:
\begin{align}
	y_t=\bigg(\begin{array}{c}{X_{0}} \\ {Y_{0}}\end{array}\bigg)+\int_{0}^{t}F (y_s)ds+\int_{0}^{t}G(y_s)d \bigg(\begin{matrix}\boldsymbol{\omega_1}(s)\\
		\boldsymbol{\omega_2}(s) 
	\end{matrix}\bigg).
\end{align}
Then the rough path $\mathbf{x}$ is generated by in canonical form $(\boldsymbol{\omega}_1,\boldsymbol{\omega}_2)$. This rough path is included in $\mathscr{C}^{0,\gamma}_g(I,\RR^{d_1+d_2})$.

\begin{remark}
 Since the diffusion coefficient associated with $\boldsymbol{\omega}_2$ is the identity map $I_d$, its Gubinelli derivative is identically zero. Consequently, the compensation term involving the second-level component of $\boldsymbol{\omega}_2$ vanishes in the definition of the rough integral in (\ref{2.7}). Hence,
$$
\int_0^t d\boldsymbol{\omega}_2(s)
=\omega_2(t)-\omega_2(0).
$$
In particular, if \(\omega_2(0)=0\), then
$$
\int_0^t d\boldsymbol{\omega}_2(s)=\omega_2(t).
$$
Therefore, the second-level component of \(\boldsymbol{\omega}_2\) is not needed to define this integral.
\end{remark}

Now, we assume that the following conditions for the coefficients of the system are fulfilled.
\begin{enumerate}
	\item[(A1)]  $f: \mathbb{R}^{n}\times \mathbb{R}^{m} \rightarrow \mathbb{R}^{n}$  is globally Lipschitz continuous with the Lipschitz constant  $C_{f}$;
	\item[(A2)] $h$  belongs to $C_{b}^{3}(\mathbb{R}^{n}, \mathcal{L}(\mathbb{R}^{d_1},\mathbb{R}^{m}))$ such that
	$$C_{h}:=\max \{\|h\|_{\infty},\|D h\|_{\infty},\|D^{2}h\|_{\infty},\|D^{3}h\|_{\infty}\}<\infty.$$
	\item[(A3)]
	$g: \mathbb{R}^{n} \times \mathbb{R}^{m} \rightarrow \mathbb{R}^{m}$  is globally Lipschitz continuous with the Lipschitz constant  $C_{g}$ and let $\|g(0,0)\|< C$.
	\item[(A4)] $\|f\|_\infty:=\sup_{(x,y) \in \mathbb R^n \times \mathbb R^m} f(x,y) <\infty $.
\end{enumerate}
\begin{lemma}
	Let {\rm (A1)-(A3)} hold. For any $X_0\in \mathbb{R}^n, Y_0\in \mathbb{R}^m$ and $T>0$, there is a unique solution $(X_t^\eps,Y_t^\eps)$ to {\rm (\ref{slowpathint})-(\ref{fastpathint})}.
\end{lemma}
\begin{proof}
	This is just the special case of Lemma \ref{solution}, we omit the proof here.
\end{proof}

\begin{lemma}\label{norm}Let {\rm (A1)-(A4)} hold. 
	The supremum and  $p$-variation norms of the solution of the slow component are estimated as follows
	\begin{align*}
	\|X^\eps\|_{\infty,[0, T]} \leq &\big(\|X_{0}\|+({\|f\|_{\infty}}{C_{f}^{-1}}+{C_{p}^{-1}}) N_{\frac{1}{4C_p C_h},[0, T],p}(\boldsymbol{\omega_1})\big)e^{4 C_{f}T}, \\
		\interleave X^\eps, R^{X^\eps}\interleave_{p-{\rm var},[0, T]}
		\leq& \big(\|X_{0}\|+({\|f\|_{\infty}}{C_{f}^{-1}}+{C_{p}^{-1}})  N_{\frac{1}{4C_p C_h},[0, T],p}(\boldsymbol{\omega_1})\big)\\
		&\quad \times  \exp{\{4 C_{f}T\}} N_{\frac{1}{4C_p C_h},[0, T],p}^{\frac{p-1}{p}}(\boldsymbol{\omega_1})-\|X_{0}\|.\end{align*}
\end{lemma}
\begin{proof}
	Let $(Y_t^\eps)_{t\in[0,T]}$ be a continuous path, we consider
	\begin{align}\label{slowpathintonly}
		X_t^\eps=& X_0+\int_{0}^{t}f( X_r^\eps,Y_r^\eps)\,dr+\int_{0}^{t}h(X_r^\eps)\,d\boldsymbol{\omega_1}(r).\end{align}
	From (A1) and (A4), one has
	$$\|f(\xi,\zeta)\|=\|f(\xi,\zeta)-f(0,\zeta)+f(0,\zeta)\|\leq C_f\|\xi\|+\sup_{\zeta\in \mathbb{R}^m}\|f(0,\zeta)\|$$
	where $\sup_{\zeta\in \mathbb{R}^m}\|f(0,\zeta)\| \leq \|f\|_{\infty}$. Consequently, applying the solution norm estimates in \cite[Theorem 3.8]{duc2020controlled} to \eqref{slowpathintonly}, with the bound on the drift at the origin replaced by $\|f\|_\infty$, yields the claimed estimates for $\|X^\eps\|_{\infty,[0,T]}$ and $\interleave X^\eps,R^{X^\eps}\interleave_{p\text{-}\mathrm{var},[0,T]}$.   
\end{proof}

\begin{lemma}\label{xbound}
	Let {\rm (A1)-(A4)}  hold. Then, for all $T>0$, we have
	$$\| X^\eps\|_{\infty}+ \interleave X^\eps, R^{X^\eps}\interleave_{p-{\rm var},[0, T]} \le C_1$$
	where $C_{1}>0$ depends on $C_p, C_f, C_h, \|f\|_{\infty},p,\|X_0\|,\interleave \boldsymbol{\omega_1}\interleave_{p-{\rm var},[0, T]} $ but independent of $\eps$.	
\end{lemma}
\begin{proof}
	Apply Lemma \ref{norm}, it is easy to show the desired estimate.
\end{proof}
\begin{remark}
	Since the boundedness of the function $f$, the parameter $1/\eps$ in the fast component will not affect the estimates of $\| X^\eps\|_{\infty}$ and $\interleave X^\eps, R^{X^\eps}\interleave_{p-{\rm var},[0, T]}$, thus it is natural that $ C_1$ is independent of $\eps$.
\end{remark}
\begin{lemma}\label{xdelta}
Suppose that assumptions {\rm (A1)-(A4)} hold. Let $\frac{1}{p}\in(\frac{1}{3},\gamma)$ with $\gamma<H_{\min}:=\min\{H,\hat{H}\}$. Then, for any $0\le s<t\le T$, the following estimate holds:
	\[\|X^\eps_t-X^\eps_s\|\le  C_2 (t-s)^{\gamma}\]
	where $C_2>0$ depends on $C_1, C_p,  C_f, C_h,\|f\|_{\infty}, \interleave \boldsymbol{\omega_1}\interleave_{\gamma,[0,T]}$ but independent of $\eps$.
\end{lemma}
\begin{proof} Consider the $\gamma$-rough path $\boldsymbol{\omega_1}=(\omega_1,\bbomega_1)  \in \mathscr{C}_{g}^{0,\gamma}([0, T], \mathbb{R}^{d_1})$ and
	denote \begin{align*}
		X^\eps_{s,t}:=X^\eps_t-X^\eps_s, \quad
		\omega_{1,s,t}:=\omega_{1}(t)-\omega_{1}(s)
	\end{align*}
	then,
	we begin with (\ref{slowpathint}) and (\ref{inq-1}) such that
	\begin{align*}
		\|X^\eps_{s, t}\|
		\leq  & \bigg\|\int_{s}^{t} f(X^\eps_{u},Y^\eps_{u}) d u\bigg\|+\bigg\|\int_{s}^{t}h(X^\eps_{u})d \boldsymbol{\omega_1}(u)\bigg\| \\
		\leq & \|f\|_{\infty} (t-s)+\|h(X^\eps_{s})\|\|\omega_{1,s, t}\|+\|D h(X^\eps_{s}) h(X^\eps_{s})\|\|\bbomega_{1,s, t}\| \\
		& +C_{p}\big(\interleave \omega_{1}\interleave_{p-{\rm var},[s, t]}\interleave R^{h(X^\eps)} \interleave_{q-{\rm var},[s, t]^{2}}\\
		&\quad +\interleave D h(X^\eps_{s}) h(X^\eps_{s}) \interleave_{p-{\rm var},[s, t]}\interleave \bbomega_{1}\interleave_{q-{\rm var},[s, t]^{2}}\big).
	\end{align*}

	Since
	\begin{align}\label{rh}
		R_{s, t}^{h(X^\eps)} =&h(X^\eps)_{s,t}-D h(X^\eps_s) h(X^\eps_s) \omega_{1,s,t}
	\end{align}
	and using (A2), we have
	\begin{align*}
		\|R_{s, t}^{h(X^\eps)}\| \leq &\bigg\|\int_{0}^{1} D h(X^\eps_s+rX^\eps_{s, t}) R_{s, t}^{X^\eps} d r\bigg\|\\
		&+\bigg\|\int_{0}^{1}(D h(X^\eps_s+rX^\eps_{s, t})-D h(X^\eps_s))h(X^\eps_s) \omega_{1,s,t} d r\bigg\|\\
		\leq & C_{h}\| R^{X^\eps}_{s,t}\|+\frac{1}{2} C_{h}^{2}\| X^\eps_{s,t}\|\|\omega_{1,s,t}\|\\
		&\leq C_{h}\interleave R^{X^\eps}\interleave_{q-{\rm var},[s, t]^{2}}+\frac{1}{2} C_{h}^{2}\interleave X^\eps \interleave_{p-{\rm var},[s, t]}\interleave \omega_{1}\interleave_{p-{\rm var},[s, t]}.
	\end{align*}

	Then, by Lemma \ref{control}, one has
	\begin{align*}
		\interleave R^{h(X^\eps)}\interleave_{q-{\rm var},[s, t]^{2}} \leq &  C_{h}L_{2,q}^{1/q}\interleave R^{X^\eps}\interleave_{q-{\rm var},[s, t]^{2}}\\
		&+\frac{1}{2} C_{h}^{2} L_{2,q}^{1/q} \interleave X^\eps \interleave_{p-{\rm var},[s, t]}\interleave \omega_{1}\interleave_{p-{\rm var},[s, t]}
	\end{align*}
	where $L_{2,q}^{1/q} \geq 1$ is defined in Lemma \ref{control} for $n=2$.
	
	Next, one has $$
	\interleave D h(X^\eps) h(X^\eps)\interleave_{p-{\rm var},[s, t]} \leq  2 C_{h}^{2}\interleave X^\eps \interleave_{p-{\rm var},[s, t]}.$$
	
	Thus,
	we have
	\begin{align*}
		\|X^\eps_{s, t}\|&
		\leq  \|f\|_{\infty} (t-s)+C_{h}\interleave \omega_{1}\interleave_{p-{\rm var},[s, t]}+C_{h}^{2}\interleave \bbomega_{1}\interleave_{q-\operatorname{var},[s, t]^{2}}\\
		&+C_p \big(\interleave \bbomega_{1}\interleave_{q-{\rm var},[s, t]^{2}} 2C_{h}^{2}\interleave X^\eps\interleave_{p-{\rm var},[s, t]}+ \interleave \omega_{1}\interleave_{p-{\rm var},[s, t]}\\\
		&\quad  \times \big( C_{h}L_{2,q}^{1/q} \interleave R^{X^\eps}\interleave_{q-{\rm var},[s, t]^{2}}+\frac{1}{2} C_{h}^{2}L_{2,q}^{1/q}\interleave \omega_{1}\interleave_{p-{\rm var},[s, t]}\interleave X^\eps \interleave_{p-{\rm var},[s, t]}\big) \big)\\
		\leq & \|f\|_{\infty} (t-s)\\
		&+4 C_p L_{2,q}^{1/q} (C_{h}^{2}\|\boldsymbol{\omega_1}\|_{p-v a r,[s, t]}^{2} \vee C_{h}\|\boldsymbol{\omega_1}\|_{p-v a r,[s, t]})(1+\interleave X^\eps, R^{X^\eps}\interleave_{p-{\rm var},[s, t]}).
	\end{align*}
	
	Thus,  given $\frac{1}{p}\in(\frac13,\gamma), $ by  Lemma \ref{solution} and the fact that $$\interleave \boldsymbol{\omega_1} \interleave_{p-{\rm var},[s, t]}\leq \interleave \boldsymbol{\omega_1} \interleave_{\gamma,[s, t]}(t-s)^\gamma$$ one has
	$$\|X^\eps_{s, t}\| \le C_2  (t-s)^{\gamma}$$
	where $C_2>0$ depends on $L_{2,q}^{1/q}, C_1, C_p,  C_f, C_h,\|f\|_{\infty}, \interleave \boldsymbol{\omega_1}\interleave_{\gamma,[0,T]}$ but independent of $\eps$.
\end{proof}

\subsection{A stationary Ornstein-Uhlenbeck-process for the FBM}\label{s3-2}
For a positive matrix $A$, we consider the equation
\begin{equation}\label{eq1}
	dZ=-A Z dt+d\omega_2
\end{equation}
or equivalent
\begin{equation}\label{eq1a}
	Z(t)=Z(r)+\int_r^t A Z(q)dq+\omega_2(t)-\omega_2(r),\quad r\not=t.
\end{equation}
on the metric dynamical system $(\Omega_2,\fF_2,\PP_{\hat H},\theta)$. Let us define
\[
Z(\omega_2)=Z^1(\omega_2)=\int_{-\infty}^0e^{Ar} d\omega_2(r)
\]
which is defined to be the limit of Riemann integrals, see Cheridito et al. \cite[Proposition A1]{cheridito2003fractional} ) and equal to
\[
-A\int_{-\infty}^0 e^{Ar }\omega_2(r)dr.
\]
For $\omega_2\in \Omega_2$ this integral is well defined.
It is easy to check that
\[
r\mapsto Z(\theta_r\omega_2)
\]
is a stationary solution to (\ref{eq1}).\\
Consider
\[
dZ^\eps=-\frac{1}{\eps}A Z^\eps dt+d\omega_{2,\eps}
\]
with stationary solution
\[
r\mapsto -\frac{A}{\eps}\int_{-\infty}^0 e^{\frac{A}{\eps}r }\omega_{2,\eps}(r)dr.
\]
Note that $\omega_2\in\Omega_2$ if and only if $\omega_{2,\eps}\in\Omega_2$.
\begin{lemma}\label{l2}
	We have for $\omega_2\in\Omega_2,\,\eps>0,\,t\in\RR$
	\[
	Z^\eps(\theta_t\omega_2)=Z(\theta_{\frac{t}{\eps}}\omega_2)=Z^1(\theta_{\frac{t}{\eps}}\omega_2).
	\]
\end{lemma}
\begin{proof}
 For the proof we refer to Pei et al. \cite{pei2023almost}.   
\end{proof}

\begin{lemma}\label{l1}
The following properties hold:
\begin{enumerate}
    \item[(1)] For every $\eta<\hat H$, the mapping
    \[
    t\mapsto Z(\theta_t\omega_2)
    \]
    is $\eta$-H\"older continuous on every compact interval $[T_1,T_2]$, where $T_1<T_2$. Moreover, the corresponding set of $\omega_2$ can be chosen to be $(\theta_t)_{t\in\mathbb R}$-invariant.

    \item[(2)] We have
    \[
    \mathbb{E}\left[\sup_{t\in[0,1]}\|Z(\theta_t\omega_2)\|\right]<\infty.
    \]

    \item[(3)] For every $T>0$, there exists a $(\theta_t)_{t\in\mathbb R}$-invariant set $\Omega_0$ of full measure such that, for every $\omega_2\in\Omega_0$,
    \[
    \sup_{s\in[0,T]}\|Z^\varepsilon(\theta_s\omega_2)\|
    =o(\varepsilon^{-1}),
    \qquad \varepsilon\to0.
    \]
\end{enumerate}
\end{lemma}
\begin{proof}
	(1) The continuity of $Z(\theta_t\omega_2)$ and the $\eta$-H\"older-continuity follows by the $\eta$-H\"older continuity of $\omega_2$, see (\ref{eq1a}). In particular we obtain the $\eta$-H\"older-continuity of $Z(\theta_t\omega_2)$ on any interval $[-K,K],\,K \in\NN$ what allows us to conclude that the existence of a $(\theta_t)_{t\in\RR}$ invariant set full measure of elements $\omega_2$ on which we have the desired H\"older continuity.
\end{proof}

\subsubsection{A short introduction of random dynamical systems}
We now provide a brief introduction to random dynamical systems; for a more detailed discussion, we refer the reader to \cite[Section 2]{schmalfuss1998random} or Caraballo et al.  \cite[Section 2]{caraballo2004exponentially}.

Let $(\Omega,\fF,\PP)$ be a probability space. In addition, let $\mathbb B$ be a separable Banach space. On $\Omega$ a measurable
flow $\theta$   so that
\[
\theta_{t+s}=\theta_t\theta_s=\theta_t\theta_s,\quad s, t\in\RR,\quad \theta_0={\rm id}_\Omega.
\]
and preserving the measure $\PP$: $\theta_t \PP=\PP$ for all
$t\in\RR$ is defined. Then $(\Omega,\fF,\PP,\theta)$  is called a metric
dynamical system. For our application we need that this metric dynamical system is
ergodic. A random variable $X\ge 0$ is called tempered if
\[
\lim_{t\to\pm\infty}\frac{\log^+ X(\theta_t\omega)}{|t|}=0.
\]
A family of sets
$(C(\omega))_{\omega\in\Omega},\,C(\omega)\not=\emptyset$ and closed is
called tempered
random set if
${\rm distance}_{\mathbb B}(y,C(\omega))$ for all $y\in \mathbb B$ is
measurable, it is called tempered if
\[
X(\omega)=\sup_{x\in C(\omega)}\|x\|_{\mathbb B}
\]
is tempered. We note that for every random set there exists
a sequence
of random variables $(x_n)_{n\in\mathbb{N}}$ so that
\[
C(\omega)=\overline{\bigcup_{n\in\mathbb{N}}\{x_n(\omega)\}}.
\]
\\
A measurable mapping
\[
\phi:\RR^+\times \Omega\times \mathbb B\to \mathbb B
\]
is called a  random dynamical system (RDS) if the cocycle property holds:
\begin{align*}
	\phi(t,\theta_\tau\omega,\cdot)\circ \phi(\tau,\omega,\cdot)= &
	\phi(t+\tau,\omega,\cdot),\quad t,\tau\ge 0,\omega\in \Omega, \\
	\phi(0,\omega,\cdot)=&{\rm id}_{\mathbb B} ,\quad\omega\in\Omega.
\end{align*}
A random variable $Y:\Omega\to \mathbb B$ is called random fixed point of the
RDS $\phi$ if
\[
\phi(t,\omega,Y(\omega))=Y(\theta_t\omega), \quad t\ge
0,\,\omega\in\Omega.
\]
Now we present sufficient conditions for the existence of a random fixed
point.
\begin{lemma}\label{l6}
	Suppose that the RDS $\phi$ has a random forward
	invariant closed set
	$C$ which is tempered:
	\[
	\phi(t,\omega,C(\omega))\subset C(\theta_t\omega)\quad
	\text{for}\quad t\ge 0,\,\omega\in\Omega.
	\]
	Let
	\[
	k(\omega)=\sup_{x\not=y\in
		C(\omega)}\log\bigg(\frac{\|\phi(1,\omega,x)-\phi(1,\omega,y)\|}{\|x-y\|}\bigg)
	\]
	so that $\mathbb{E} k<0$. The random variable
	\begin{eqnarray}\label{phi}
		\omega\mapsto\sup_{t\in
			[0,1]}\|\phi(t,\theta_{-t}\omega,y(\theta_{-t}\omega))\|
	\end{eqnarray}
	is assumed to be tempered for any measurable selector $y$ from
	$C$. Then the
	RDS $\phi$ has a random fixed point $Y(\omega)\in
	C(\omega)$ which is unique. In addition $\|Y(\omega)\|$ is tempered.
	This random fixed point is pullback and forward
	attracting:
	\begin{equation}\label{eq6}
		\lim_{t\to\infty}\|\phi(t,\theta_{-t}\omega,y(\theta_{-t}\omega))-Y(\omega)\|=0,
		\,\lim_{t\to\infty}\|\phi(t,\omega,y(\omega))-Y(\theta_t\omega)\|=0
	\end{equation}
	with exponential speed for every measurable selector
	$y(\omega)\in C(\omega),\,\omega\in\Omega$. For this lemma we refer to Schmalfuss \cite[Theorem 2.2]{schmalfuss1998random}.
\end{lemma}

\subsection{Random fixed points for the fast equation with frozen solution of the slow equation}
To obtain the random fixed points of the fast component, we assume further that
\begin{enumerate}
	\item[(A5)] 
    Let $A$ be a positive definite matrix in $\mathbb{R}^{m\times m}$, and denote by $\lambda_A>0$ its smallest eigenvalue. Then
	\[
	(Ay,y)\ge \lambda_A  \|y\|^2\quad \text{for all } y\in\RR^m,
	\] and $g$ has the following relation
	\[g(x,y)=-Ay+\tilde g(x,y)\] where
	$\tilde g : \mathbb{R}^{n} \times \mathbb{R}^{m} \rightarrow \mathbb{R}^{m}$ is Lipschitz continuous with $\lambda_A>L_{\tilde g}>0$ such that
	\[
	\|\tilde g(x_1,y_1)-\tilde g(x_2,y_2)\|\le L_{\tilde g}(\|x_1-x_2\|+\|y_1-y_2\|).
	\]
\end{enumerate}

We would like to deal with random fixed points of the RDS generated by the equation (\ref{eq2}) for the Banach space $\mathbb B=\RR^m$:
\begin{equation}\label{eq2}
	dy=\frac{1}{\eps}(-Ay+\tilde g(x,y))dt +d\omega_{2,\eps},\quad y(0)=y_0
\end{equation}
for every $x\in \RR^n$.
An RDS is often generated  by the solution of an SDE. For our case we consider
the equation
\[
\frac{dy^{x,\eps}(t)}{dt}=\frac{1}{\eps}(-Ay^{x,\eps}(t)+\tilde
g^\eps(x,y^{x,\eps}(t),\theta_t\omega_2)),\quad \tilde
g^\eps(x,y,\omega_2))=\tilde g(x,y+Z^\eps(\omega_2)).
\]
The solutions of this equation generate an RDS
$\tilde\phi^{x,\eps}$. Consider the conjugated RDS
\begin{align*}
	\phi^{x,\eps}(t,\omega_2,y_0)&=T^\eps(\theta_t\omega_2,\tilde
	\phi^{x,\eps}(t,\omega_2,(T^{\eps})^{-1}(\omega_2,y_0))), \\
	T^\eps(\omega_2,y)&=y+Z^\eps(\omega_2),\quad
	(T^{\eps})^{-1}(\omega_2,y))=y-Z^\eps(\omega_2).
\end{align*}
Then $\phi^{x,\omega}(\cdot,\omega_2,y_0)$ presents a solution  of (\ref{eq2}).

\begin{lemma}\label{l3}
	Let $C^x(\omega_2)$ be the ball with center $0$ and square radius
	\[
	\rho^x(\omega_2)^2=2\int_{-\infty}^0e^{\frac{(\lambda_A  -C_{\tilde g}-\mu )r}{\eps}}
	\frac{C_{\hat g}}{\eps\mu}^2(3\|x\|^2+3\|Z^\eps(\theta_r\omega_2)\|^2+3\|\tilde g(0,0,0)\|^2)dr
	\]
	where $0< \mu<\lambda_A-C_{\tilde g}$.
	Then we have
	\[
	\tilde \phi^{x,\eps}(t,\omega_2,C^x(\omega_2))\subset C^x(\omega_2)(\theta_t\omega_2),\quad t\ge 0.
	\]
	In addition, $\tilde \phi^{x,\eps}$ has a random fixed point $\tilde Y_F^\eps(\cdot,x)$ in $C^x(\omega_2)$.
\end{lemma}
\begin{proof}
	We have
	\begin{align*}
		\frac{d}{dt}\|y^{x,\eps}(t)\|^2\le &-\frac{2}{\eps} \lambda_A
		\|y^{x,\eps}(t)\|^2+\frac{2}{\eps}(\tilde
		g(x,y^{x,\eps}(t),\theta_t\omega_2),y^{x,\eps}(t))\\
		\le &-\frac{2}{\eps} \lambda_A  \|y^{x,\eps}(t)\|^2+\frac{2}{\eps}\|\tilde
		g(x,y^{x,\eps}(t),\theta_t\omega_2)-\tilde
		g(x,0,\theta_t\omega_2)\|\|y^{x,\eps}(t)\|\\
		&+\frac{2}{\eps}\|\tilde g(x,0,\theta_t\omega_2)\|\|y^{x,\eps}(t)\|\\
		\le& -\frac{2(\lambda_A -C_{\tilde g})}{\eps}
		\|y^{x,\eps}(t)\|^2+\frac{2}{\eps}\|\tilde
		g(x,0,\theta_t\omega_2)\|\|y^{x,\eps}(t)\|\\
		\le& -\frac{2(\lambda_A -C_{\tilde g})}{\eps}
		\|y^{x,\eps}(t)\|^2+\frac{2}{\eps}\|\tilde
		g(x,0,\theta_t\omega_2)-\tilde
		g(0,0,0)\|\|y^{x,\eps}(t)\|\\
		&+\frac{2}{\eps}\|\tilde
		g(0,0,0)\|)\|y^{x,\eps}(t)\|.
	\end{align*}
	Note that
	\begin{align*}
		2\|\tilde   g(x,0,\theta_t\omega_2)-\tilde
		g(0,0,0)\|& \|y^{x,\eps}(t)\|+2\|\tilde
		g(0,0,0)\|)\|y^{x,\eps}(t)\| \\
		& \le  \frac{C_{\tilde g}}{\mu}^2(\|x\|+\|Z^\eps(\theta_r\omega_2)\|+\|\hat
		g(0,0,0)\|)^2+\mu\|y^{x,\eps}(t)\|^2.
	\end{align*}
	We obtain
	that the ball $C^x(\omega_2)$ with center zero and square radius
	\begin{align}\label{eq7}
		\begin{split}
			\frac{d\hat Y^\eps(t,\omega_2,x)}{dt}=&-\frac{\lambda_A-C_{\tilde g}-\mu}{\eps}\hat Y^\eps(t,\omega_2,x)\\
			&
			+\frac{C_{\tilde g}}{\mu\eps}^2(\|x\|+\|Z^\eps(\theta_t\omega_2)\|+\|\hat
			g(0,0,0)\|)^2+\frac{1}{\eps}\mu.
		\end{split}
	\end{align}
	The variation of constants method and a comparison argument show that $$\|y^{\eps,x}(t,\omega_2)\|^2\le \hat Y^\eps(t,\omega_2,x).$$ In addition (\ref{eq7}) has the unique random and tempered fixed point $\rho^x(\omega_2)/2$. We can conclude that $\|\tilde Y_F^\eps(\omega_2,x)\|^2\le \rho^x(\omega_2)$, see  Chueshov and
	Schmalfuss \cite[Theorem 3.1.23]{chueshov2020synchronization}.\\
	We have for $y_i^{x,\eps}=\tilde \phi(t,\omega_2,y_i)$ for $y_i\in \RR^m$
	\[
	\frac{d\|y^{x,\eps}_1(t)-y^{x,\eps}_2(t)\|^2}{dt}\le\frac{1}{\eps}(-2\lambda_A+2C_{\hat
		g})\|y^{x,\eps}_1(t)-y^{x,\eps}_2(t)\|^2.
	\]
	Then the Gronwall-lemma gives the contraction condition. The other
	condition of the fixed point theorem follows  from the forward invariance of $\tilde \phi^{x,\eps}$:
	\[
	\tilde \phi^{x,\eps}(t,\theta_{-t}\omega_2,y(\theta_{-t}\omega_2))\in
	C^x(\omega_2)
	\]
	for every measurable selector $y$ of $C^x(\omega_2)$. But $C^x(\omega_2)$ has a tempered radius.
\end{proof}
\begin{remark}\label{3.11}
	(1) We note that the radius $\rho^x(\omega_2)$ of $C^x(\omega_2)$ can be chosen independently of $\eps$
	which follows by the simple integral substitution $r^\prime =r/\eps$
	and Lemma \ref{l2}.\\
	(2) The RDS $\phi^{x,\eps}$ has the random fixed point
	\[
	Y_F^\eps(\omega_2,x)=\tilde Y_F^\eps(\omega,x)+Z^\eps(\omega_2)
	\]
	contained in $C^x(\omega_2)(\omega_2)+Z^\eps(\omega_2)$ with center $Z^\eps(\omega_2)$ and square radius $\rho^x(\omega_2)^2$.\\
	(3) We have
	\[
	\tilde Y_F^\eps(\theta_t\omega_2,x)=\tilde Y_F^1(\theta_\frac{t}{\eps}\omega_2,x)=\tilde Y_F(\theta_\frac{t}{\eps}\omega_2,x).
	\]
	The same holds for $Y_F^\eps$.\\
	(4) The contraction constant $k$ of Lemma \ref{l6} is independent of $x$ and $\omega_2$ and
	holds on $\RR^m$ because the random set $C^x$ is pullback absorbing. Then we can conclude that the random fixed point is
	unique in $\RR^m$.\\
	(5) $Y^\eps_F(\omega_2,x), \tilde Y^\eps_F(\omega_2,x)$ depend Lipschitz-continuously on $x$ with Lipschitz-constant $C_{\tilde g}/(\lambda_A-C_{\tilde g})$.
\end{remark}
For the proof of this remark we refer to Pei et al. \cite{pei2023almost}.

\subsection{An ergodic theorem}
Now, we formulate an ergodic theorem. 
Define
\begin{eqnarray}\label{conave1}
	\bar f (x)=\mathbb{E}[f(x,Y^1_F(\omega_2,x))].
\end{eqnarray}
Then, according to condition (A1), 
$\bar f$ is also Lipschitz continuous (see \cite[Lemma 4.10]{pei2023almost}); that is, for any 
$x,y \in \mathbb R^n$, we have
\begin{equation}\label{barf}
    \|\bar{f}(x)-\bar{f}(y)\| \leq C \|x-y\|.
\end{equation}
\begin{lemma}\label{l5}
	There exists a $(\theta_t)_{t\in\RR}$ invariant set of full measure so that for every
	$\omega_2$ from this set and $x\in \RR^m$ we have
	\[
	\lim_{T\to\pm \infty}\bigg\| \frac{1}{T}\int_0^T (f(x, Y_F^1(\theta_r\omega_2,x))-\bar
	f(x)) dr\bigg\| =0.
	\]
\end{lemma}

\begin{proof}
	Let $\Omega_x\in \fF$ be a $(\theta_t)_{t\in\RR}$-invariant set of full
	measure  so that
	\[
	\lim_{T\to\pm\infty} \frac{1}{T} \int_0^T(f(x,Y_F^1(\theta_r\omega_2,x))-\bar f(x)) dr=0
	\]
	for $x\in \RR^n$. Let $D\subset \RR^n$ be a dense countable set. Then
	$\bigcap_{x\in D}\Omega_x$ has full measure and is
	$(\theta_t)_{t\in\RR}$-invariant.  We now choose an $x\not\in D$ and a
	sequence $(x_n)_{n\in\NN}$ in $D$ for an arbitrary $\zeta>0$ so that for
	sufficiently large $\tilde n$ we have that $\|x_{\tilde n}-x\|<\zeta/2$.
	By the Lipschitz continuity of $x\mapsto \bar f(x), \quad\text{and }x\mapsto Y_F^1(\omega,x)$
	we have that
	\begin{align*}
		\bigg\| \frac{1}{T}&\int_0^T (f(x, Y_F^1(\theta_r\omega_2,x))-\bar f(x))
		dr\bigg\|\\
		\le& \bigg\| \frac{1}{T}\int_0^T (f(x, Y_F^1(\theta_r\omega_2,x))dr-\frac{1}{T}\int_0^T (f(x_{\tilde n}, Y_F^1(\theta_r\omega_2,x_{\tilde n})))dr\bigg\|\\
		&+\bigg\| \frac{1}{T}\int_0^T  (f(x_{\tilde n}, Y_F^1(\theta_r\omega_2,x_{\tilde n}))-\bar f(x_{\tilde n}))dr\bigg\|
		+\bigg\| \frac{1}{T}\int_0^T  (\bar f(x_{\tilde n}))-\bar f(x))dr\bigg\|.
	\end{align*}
	The first and the last term of the right hand side of this inequality can by made smaller than $C\|x-x_{\tilde n}\|\le C\zeta/2$
	where $C$ bounds the Lipschitz-constants of $f,\,\bar f$.
	The other term   can be made smaller than $\zeta/2$ for large $|T|$. Hence choosing $\zeta$ sufficiently small the left hand side can be made arbitrarily small.
\end{proof}
\section{Almost sure averaging for fast-slow SDEs}\label{s4}
\subsection{Some a-priori estimates of the fast component}
Following the discretization techniques inspired by
Khasminskii in
\cite{khasminskii1968on},  we divide $[0,T]$ into intervals of
size $\delta$,
where $\delta \in(0,1)$ is a fixed number.
Note that (A5) ensures that the group $\Phi_A(t) = e^{-At}, t\in \mathbb{R}$ generated by $A$ satisfies the following properties
\begin{align}
	\|\Phi_A(t)Y\|& \leq e^{-\lambda_{A} t} \|Y\|, \, {\rm for} \, t\geq 0, Y\in \mathbb{R}^m
\end{align}
where $\lambda_A>0$ defined in (A5). Then,
we construct an auxiliary process $\hat{Y}^{\eps}$ and for $t\in[k\delta, (k+1)\delta),$
\begin{eqnarray}
	\label{yhat} \hat{Y}^{\eps}_t&=&\Phi_{{A}/{\eps}} (t-k\delta)\hat{Y}^{\eps}_{k\delta}+\frac{1}{\eps}\int_{k\delta}^t\Phi_{{A}/{\eps}}(t-r)
	\tilde g(X^{\eps}_{k\delta},\hat{Y}^{\eps}_r)\,dr\cr
	&&+\int_{k\delta}^t \Phi_{{A}/{\eps}}(t-r)\,d\omega_{2,\eps}(r)
\end{eqnarray}
i.e. for  $t\in [0,T],$
\begin{eqnarray}
	\quad \label{yhat0}\hat{Y}^{\eps}_t=\Phi_{{A}/{\eps}} (t)\hat{Y}^{\eps}_0+\frac{1}{\eps}\int_0^t \Phi_{{A}/{\eps}}(t-r) \tilde g(X^{\eps}_{r_{\delta}},\hat{Y}^{\eps}_r)\,dr+\int_0^t\Phi_{{A}/{\eps}}(t-r)\,d\omega_{2,\eps}(r)
\end{eqnarray}
where $r_{\delta}=\lfloor r / \delta\rfloor \delta$ is the
nearest
breakpoint preceding $r$.

\begin{lemma}\label{ybound}
	For any solution $Y^\eps$ of {\rm(\ref{fastpath})}, we have
	\begin{eqnarray*}
		\|Y^\eps\|_{\infty}+\|\hat Y^\eps\|_{\infty}
		\le C_3 +C_4 o(\eps^{-1}).
	\end{eqnarray*} where $C_3$ may depend on $\lambda_A, C_1, C_{\tilde g}, \|g(0,0)\|, \|Y_{0}\|$ and $C_4$ may depend on  $\lambda_A, C_{\tilde g}$  and $\interleave \omega_2 \interleave_{\gamma,[0,T]}$.
\end{lemma}
\begin{proof}
	For $t\in[0,T]$, from (\ref{fastpath}), one has
	\begin{align*}
		Y^\eps_t=&\,\Phi_{ A/\eps}(t)
		Y_0+\frac{1}{\eps}\int_{0}^{t}\Phi_{ A/\eps}(t-r)\tilde g(
		X^\eps_r,Y^\eps_r)\,
		dr+\int_{0}^{t}\Phi_{ A/\eps}(t-r)\,d\omega_{2,\eps}(r)\cr=&\,\Phi_{ A/\eps}(t)
		(Y_0-Z^\eps(\omega_2))+Z^\eps(\theta_t\omega_2)+\frac{1}{\eps}\int_{0}^{t}\Phi_{ A/\eps}(t-r)\tilde g(
		X^\eps_r,Y^\eps_r)\,dr.
	\end{align*}
	
	Then, we have
	\begin{eqnarray*} \|Y^\eps_t\|&\le
		&\|\Phi_{ A/\eps}(t)\|\|Y_0-Z^\eps(\omega_2)\|+\|Z^\eps(\theta_t\omega_2)\|+\bigg\|\frac{1}{\eps}\int_{0}^{t}\Phi_{ A/\eps}(t-r)\tilde g(
		X^\eps_r,Y^\eps_r)\,dr\bigg\|\cr
		&\le & e^{-\lambda_A t/\eps }\|Y_0-Z^\eps(\omega_2)\|+\|Z^\eps(\theta_t\omega_2)\|\cr
		&&+
		\frac{1}{\eps}\int_{0}^te^{-\lambda_A(t-r)/\eps}\big(\|g(0,0)\|+C_{\tilde g}(\|X^{\eps}_r\|+\|Y^{\eps}_r\|)\big)\,dr.
	\end{eqnarray*}
	
	By \cref{xbound} and (A5), it follows
	\begin{eqnarray*}
		\sup_{t\in[0,T]}\|Y^\eps_t\|
		&\le
		&\|Y_0\|+2\sup_{t\in[0,T]}\|Z^\eps(\theta_t\omega_2)\|\cr&&+\sup_{t\in[0,T]}
		\frac{1}{\eps}\int_{0}^te^{- \lambda_A (t-r) / \eps}(\|g(0,0)\|+C_{\tilde g}(\|X^{\eps}_r\|+\|Y^{\eps}_r\|))\,dr\cr
		&\le & \|Y_0\|+2
		\sup_{t\in[0,T]}\|Z^\eps(\theta_t\omega_2)\|+ \lambda_A^{-1}(\|g(0,0)\|+C_{\tilde g} C_1 +C_{\tilde g}\sup_{t\in[0,T]}\|Y^{\eps}_t\|).
	\end{eqnarray*}
	Then, by $\lambda_A>C_{\tilde g}$ and Lemma \ref{l1}, we have
	\begin{eqnarray*}
		\|Y^\eps\|_{\infty}
		\le C_3 +C_4 o(\eps^{-1}).
	\end{eqnarray*} where $C_3$ may depend on $\lambda_A, C_1, C_{\tilde g}, \|g(0,0)\|, \|Y_{0}\|$ and $C_4$ may depend on  $\lambda_A, C_{\tilde g},$  and $\interleave \omega_2 \interleave_{\gamma,[0,T]}$.
	The estimate for  $\|\hat
	Y^\eps\|_{\infty}$ can
	be obtained in a similar way.
\end{proof}
\begin{lemma}\label{y-yhat}
	For the solution $\hat Y^\eps$ of Eq. {\rm (\ref{yhat})} and the solution $Y^\eps$ of Eq. {\rm (\ref{fastpath})},
	there exist $\delta>0$, $\gamma<H_{\min}$ and $C_5>0$ such that
	\begin{eqnarray*}
		\int_{k\delta}^{(k+1)\delta}\|Y^\eps_r-\hat Y^\eps_r\|\,dr
		&\le& C_5 \delta^{1+\gamma},\quad {\rm for \, any}\,\,0\leq k \leq \lfloor T / \delta\rfloor-1\cr
		\int_{t_\delta}^{t}\|Y^\eps_r-\hat Y^\eps_r\|\,dr
		&\le& C_5 \delta^{1+\gamma},\quad {\rm for \, any}\, \, 0\leq t \leq T
	\end{eqnarray*}
	hold, where $C_5$ is a constant which is independent of $\eps$ and $\delta$.
\end{lemma}
\begin{proof} For $r\in[k \delta,(k+1) \delta], 0\leq k \leq \lfloor T / \delta\rfloor-1$,
	one has
	\begin{eqnarray*}
		\|Y^\eps_r-\hat Y^\eps_r\|&\le &e^{-\lambda_A (r-k \delta)/\eps  }\|
		Y^\eps_{k \delta}-\hat Y^\eps_{k \delta}\|\cr
		&&+
		\bigg\|\frac{1}{\eps}\int_{k \delta}^r \Phi_{{A}/{\eps}}(r-v)(\tilde {g}(X^{\eps}_v,Y^{\eps}_v)-\tilde {g} (X^{\eps}_{v_{\delta}},\hat
		Y^{\eps}_v))\,dv\bigg\|\cr
		&\le
		&C e^{-\lambda_A(r-k \delta)/\eps  }(\|
		Y^\eps\|_{\infty}+\|\hat Y^\eps\|_{\infty})\cr
		&&+\frac{C_{\tilde {g}}}{\eps}\int_{k \delta}^re^{-\lambda_A(r-v)/\eps}(\|X^{\eps}_v-X^{\eps}_{v_{\delta}}\|+\|Y^{\eps}_v-\hat
		Y^{\eps}_v\|)\,dv.
	\end{eqnarray*}
	Then, multiplying both sides of the above equation by
	$e^{\lambda_A r/\eps}$, due to Lemma \ref{ybound}, we have
	\begin{eqnarray*}
		e^{\lambda_A r /\eps}\|Y^\eps_r-\hat Y^\eps_r\|
		&\le&C e^{\lambda_A k \delta/ \eps  }(C_3+C_4 o(\eps^{-1}))\cr
		&&+\frac{C_{\tilde {g}}}{\eps}\int_{k \delta}^r e^{\lambda_A v/\eps }(\|X^{\eps}_v-X^{\eps}_{v_{\delta}}\|+\|Y^{\eps}_v-\hat
		Y^{\eps}_v\|)\,dv.
	\end{eqnarray*}
	By the Gronwall-lemma \cite[p.37]{coddington1956theory}, we have
	\begin{eqnarray}\label{yy}
		\|Y^\eps_r-\hat Y^\eps_r\|
		& \leq& e^{-\lambda_A(r-k \delta)/\eps  } e^{C_{\tilde {g}}(r-k \delta)/\eps} (C_3+C_4 o(\eps^{-1}))\cr
		&&+
		\frac{C_{\tilde {g}}}{\eps}\int_{k \delta}^re^{-(\lambda_A-C_{\tilde {g}})(r-v)/\eps}\|X^{\eps}_v-X^{\eps}_{v_{\delta}}\|\,dv.
	\end{eqnarray}
	
	Integrate the above inequality from $k \delta$ to $(k+1) \delta$, by
	\cref{xdelta},  there exists $\eps^\prime_0(\delta)>0$,  for any $\eps<\eps^\prime_0(\delta)$,  such that
	\begin{align*}
		\int_{k \delta}^{(k+1) \delta}\|Y^\eps_r-\hat Y^\eps_r\|\,dr
		\leq & (C_3+C_4 o(\eps^{-1}))\int_{k \delta}^{(k+1) \delta} e^{-(\lambda_A-C_{\tilde {g}}) (r-k \delta)/\eps }\,dr \\
		&+ C_2 \delta^\gamma \frac{C_{\tilde {g}}}{\eps}
		\int_{k \delta}^{(k+1) \delta}\int_{k \delta}^re^{-(\lambda_A-C_{\tilde g})(r-v)/\eps}\,dv dr\\
		\leq&  C \big(\eps (C_3+C_4 o(\eps^{-1}))+ C_2 \delta^{1+\gamma}\big)    \leq C_5 \delta^{1+\gamma}
	\end{align*}
	where we assume that  $CC_2\leq C_5$ and $ C \eps (C_3+C_4 o(\eps^{-1})) \leq C_5 \delta^{1+\gamma}$.
	
	Similarly, for any $t\in [0,T]$, taking $k=\lfloor t / \delta\rfloor$ in (\ref{yy}) and integrating the inequality from $t_\delta$ to $t$, by
	\cref{xdelta} again, we have
	$
	\int_{t_\delta}^{t}\|Y^\eps_r-\hat Y^\eps_r\|\,dr
	\leq C_5 \delta^{1+\gamma}$
	for any $\eps<\eps^\prime_0(\delta)$, where \(t\_\delta \) is defined below (\ref{yhat}).
\end{proof}

\begin{lemma}\label{y1-y2}
	For the stationary solution and any solution of {\rm (\ref{yhat})}, there exist $\delta>0$, $\gamma<H_{\min}$ and $C_6>0$ such that
	\begin{eqnarray*}
		\int_{k\delta}^{(k+1)\delta}\|\hat Y^\eps_r-Y^\eps_F(\theta_r\omega_2,X^\eps_{k \delta})\|\,dr
		&\le& C_6 \delta^{1+\gamma},\quad {\rm for \, any}\,\, 0\leq k \leq \lfloor T / \delta\rfloor-1\cr
		\int_{t_\delta}^{t}\|\hat Y^\eps_r-Y^\eps_F(\theta_r\omega_2,X^\eps_{k \delta})\|\,dr
		&\le &C_6 \delta^{1+\gamma},\quad {\rm for \, any}\,\, 0\leq t \leq T
	\end{eqnarray*}
	hold, where $C_6$ is a constant which is independent of $\eps$ and $\delta$.
\end{lemma}
\begin{proof}
	By the Gronwall-lemma argument, (\ref{eq6}) and Remark \ref{3.11} for $r\in [k \delta, (k+1) \delta],0\leq k \leq \lfloor T / \delta\rfloor-1$, we have
	\begin{eqnarray}\label{yyy}
		\|\hat
		Y^\eps_r-Y^\eps_F(\theta_r\omega_2,X^\eps_{k \delta})\|
		\le  e^{-\frac{\lambda_A-C_{\tilde g}}{\eps} (r-k \delta)
		}\|\hat
		Y^\eps_{k \delta}-Y^\eps_F(
		\theta_{k \delta}\omega_2,X^\eps_{k \delta})\|
	\end{eqnarray}
	
	Integrate above inequality from $k \delta$ to
	$(k+1) \delta$, due to
	$\lambda_A>C_{\tilde g}$, we have
	\begin{align*}
		\int_{k \delta}^{(k+1) \delta}& \|\hat
		Y^\eps_r-Y^\eps_F(\theta_r\omega_2,X^\eps_{k \delta})\|\,dr\cr
		\le& \, C \|\hat
		Y^\eps_{k \delta}-Y^\eps_F(
		\theta_{k \delta}\omega_2,X^\eps_{k \delta})\|
		\frac{\epsilon}{\lambda_A-C_{\tilde g}}(1-e^{\frac{-(\lambda_A-C_{\tilde g})\delta}{\eps}})\\
		\le & \,C (\|\hat
		Y^\eps_{k \delta}\|+\frac{C_{\tilde g}}{\lambda_A-C_{\tilde g}}\|X^\eps_{k \delta}\|+\|Y_F^\eps
		(\theta_{k\delta}\omega_2,0)\|)\frac{\epsilon}{\lambda_A-C_{\tilde g}}\\
		\le& \,C \sup_{r\in[0,T]}(\|
		X^\eps_r\|+\|Y^\eps_r\|+\rho^{0}(\theta_r\omega_2)+\|Z^\eps(\theta_r\omega_2)\|)\frac{\epsilon}{\lambda_A-C_{\tilde g}}.
	\end{align*}
	
	Thus, by Lemmas \ref{l2},  \ref{xbound} and
	\ref{ybound}, the
	desired result is obtained.
	Similarly, for any $t\in [0,T]$, taking $k=\lfloor t / \delta\rfloor$ in (\ref{yyy}) and integrating the inequality from $t_\delta$ to $t$, by Lemmas \ref{l2},  \ref{xbound} and
	\ref{ybound} again, the
	desired result can be obtained.
\end{proof}

\subsection{Main result}
We present now the main result of this article. Denote $\bar X$ be the solution of the averaging equation
\begin{eqnarray}\label{x-ave}
	\bar X_t= X_0+\int_{0}^{t}\bar f(\bar X_r)\,dr+\int_{0}^{t}h(\bar X_r)d\boldsymbol{\omega_1}(r)
\end{eqnarray}
where the Lipschitz continuous function $\bar f$ has been given in {\rm (\ref{barf})}.

Following an argument similar to that used in Lemma \ref{xdelta}, we obtain the following estimate for the time increments of the solution to the averaged equation.
\begin{lemma} For any $0\leq s < t \leq T$, there exists a constant $ C_2 > 0 $ such that  
$$
	\|\bar X_t-\bar X_s\|\le  C_2 (t-s)^{\gamma}. $$
\end{lemma}

Before proving the main result, we first give the following lemma, which plays a key role in the proof of the averaging principle.
\begin{lemma}\label{con}
Let {\rm (A1)-(A5)} hold, and let 
$$
Q^{\eps}_{t}=\int_{0}^{t} f(X^\eps_r,Y^\eps_r) \, dr,\quad \bar Q_{t}=\int_{0}^{t} \bar f(\bar X^\eps_r) dr.$$
Then there exist $\delta >0$, $\gamma <H_{min}$ and $\gamma_0 \in (0,1)$ such that 
\begin{equation*}
			\interleave Q^{\eps}-\bar Q \interleave_{p-{\rm var}, [s,t]} \leq   C (\delta+  \delta^{\gamma_0})+C_f \int_{s}^{t}\|X^\eps_r-\bar X_r\| \, dr.
		\end{equation*}
\end{lemma}
\begin{proof}
We consider the estimates of $\|Q^\eps_{s,t}-\bar Q_{s,t}\|$ into two cases.
	
	{\bf Case~1~:~$t-s \leq\delta$.}
	We show that
	\begin{align}\label{q0}
		\begin{split}
			\interleave Q^{\eps}-& \bar Q \interleave_{p-{\rm var}, [s,t]}\\
			= & \bigg(\sup _{\mathcal{P}([s,t])} \sum_{[t_{i}, t_{i+1}] \in \mathcal{P}([s,t])}\|Q^\eps_{t_{i}, t_{i+1}}-\bar Q_{t_{i}, t_{i+1}}\|^{p}\bigg)^{\frac{1}{p}}\\
			\leq  & \bigg(\sup _{\mathcal{P}([s,t])} \sum_{[t_{i}, t_{i+1}] \in \mathcal{P}([s,t])}\bigg(\int_{t_{i}}^{t_{i+1}}\| f(X^\eps_r,Y^\eps_r)- \bar f(\bar X_r)\| \, dr\bigg)^{p}\bigg)^{\frac{1}{p}}\\
			\leq &C  \sup _{\mathcal{P}([s,t])} \sum_{[t_{i}, t_{i+1}] \in \mathcal{P}([s,t])}(t_{i+1}-t_{i})\le C \delta,
		\end{split}
	\end{align}
	where $C$ depends on $\|f\|_{\infty}$ and $T$.

    {\bf Case~2~:~$t-s > \delta$.} We begin with
	\begin{eqnarray}\label{x-x}
		\|Q^\eps_{s,t}-\bar Q_{s,t}\|
		&\le&
		\bigg\|\int_{s}^{t}(f(X^\eps_r,Y^\eps_r)-f(X^\eps_r,\hat Y^\eps_r))\,dr\bigg\|\cr
		&&+\bigg\|\int_{s}^{t}(f(X^\eps_r,\hat Y^\eps_r)-f(X^\eps_{r_{\delta}},\hat Y^\eps_r))\,dr\bigg\|\cr
		&&+\bigg\|\int_{s}^{t}(f(X^\eps_{r_{\delta}},\hat Y^\eps_r)-f(X^\eps_{r_{\delta}},Y_{F}^\eps(\theta_{r}\omega_2, X^\eps_{r_{\delta}})))\,dr\bigg\|\cr
		&&+\bigg\|\int_{s}^{t}(f(X^\eps_{r_{\delta}},Y_{F}^\eps(\theta_r\omega_2, X^\eps_{r_{\delta}}))-f(X^\eps_r,Y_{F}^\eps(\theta_{r}\omega_2, X^\eps_r)))\,dr\bigg\|\cr
		&&+\bigg\|\int_{s}^{t}(f(X^\eps_r,Y_{F}^\eps(\theta_r\omega_2,X^\eps_r))-f(\bar X_r,Y_{F}^\eps(\theta_{r}\omega_2, \bar X_r)))\,dr\bigg\|\cr
		&&+\bigg\|\int_{s}^{t}(f(\bar X_r,Y_{F}^\eps(\theta_r\omega_2, \bar X_r))-f(\bar X_{r_{\delta}},Y_{F}^\eps(\theta_{r}\omega_2, \bar X_{r_{\delta}})))\,dr\bigg\|\cr
		&&+\bigg\|\int_{s}^{t}(f(\bar X_{r_{\delta}},Y_{F}^\eps(\theta_{r}\omega_2, \bar X_{r_{\delta}}))-\bar f(\bar X_{r_{\delta}}))\,dr\bigg\|\cr
		&&+\bigg\|\int_{s}^{t}(\bar f(\bar X_{r_{\delta}})-\bar f(\bar X_r))\,dr\bigg\|
		=:\sum_{i=1}^{8}I_i.
	\end{eqnarray}
	
	Based on the Lipschitz continuity of $f,\,\bar f$, $Y_F^\eps$  and Lemma \ref{xdelta} we can estimate $I_2$, $I_4$, $I_6$, and $I_8$
	\begin{align*}
		I_2+I_4 \le & C_f \int_{s}^{t} \|X^\eps_r-X^\eps_{r_{\delta}}\| \, dr
		\le C \delta^{\gamma}(t-s),\\
		I_6+I_8 \le &   C_{\bar f} \int_{s}^{t} \|\bar X_r-\bar X_{r_{\delta}}\|\,dr
		\le C \delta^{\gamma}(t-s).
	\end{align*}
	
	By \cref{y-yhat}, there exists $\delta>0$, we have
	\begin{eqnarray*}
		I_1 &\le &  C_f \int_{s}^{(\lfloor \frac{s}{\delta} \rfloor+1)\delta}\|Y^\eps_r- \hat Y^\eps_r\| \, dr +C_f \sum_{k=\lfloor \frac{s}{\delta} \rfloor+1}^{\lfloor \frac{t}{\delta} \rfloor-1}\int_{k\delta}^{(k+1)\delta}\|Y^\eps_r- \hat Y^\eps_r\| \, dr\cr
		&&+C_f \int_{\lfloor \frac{t}{\delta} \rfloor \delta}^{t}\|Y^\eps_r- \hat Y^\eps_r\| \, dr\cr
		&\leq & C_f \frac{(t-s)}{\delta}\max_{0\leq k \leq \lfloor \frac{T}{\delta} \rfloor -1}\int_{k \delta}^{(k+1) \delta}\|Y^\eps_r-\hat Y^\eps_r\|\,dr+C_f C_5 \delta^{1+\gamma}\cr
		&\le & C \delta^\gamma (t-s)
	\end{eqnarray*}
	and by \cref{y1-y2},  we have
	\begin{eqnarray*}
		I_3 &\le & C_f \int_{s}^{t} \|\hat Y^\eps_r-Y_{F}^\eps (\theta_{r}\omega_2, X^\eps_{r_{\delta}}))\| \, dr\cr
		&\le &  C_f \int_{s}^{(\lfloor \frac{s}{\delta} \rfloor+1)\delta}\|\hat Y^\eps_r-Y_{F}^\eps (\theta_{r}\omega_2, X^\eps_{r_{\delta}}))\| \, dr\cr
		&&+C_f \sum_{k=\lfloor \frac{s}{\delta} \rfloor+1}^{\lfloor \frac{t}{\delta} \rfloor-1}\int_{k\delta}^{(k+1)\delta}\|\hat Y^\eps_r-Y_{F}^\eps (\theta_{r}\omega_2, X^\eps_{k \delta}))\| \, dr\cr
		&&+C_f \int_{\lfloor \frac{t}{\delta} \rfloor \delta}^{t}\|\hat Y^\eps_r-Y_{F}^\eps (\theta_{r}\omega_2, X^\eps_{r_{\delta}}))\| \, dr\cr
		&\leq & C_f \frac{(t-s)}{\delta}\max_{0\leq k \leq \lfloor \frac{T}{\delta} \rfloor -1}\int_{k \delta}^{(k+1) \delta}\|\hat Y^\eps_r-Y_{F}^\eps (\theta_{r}\omega_2, X^\eps_{k \delta}))\|\,dr+C_f C_6 \delta^{1+\gamma}\cr
		&\le &	 C \delta^\gamma (t-s).
	\end{eqnarray*}
	
	Then, for $I_5$, we have
	\begin{align*}
		I_5 \le C_f \int_{s}^{t} \|X^\eps_r-\bar X_r\|\,dr.
	\end{align*}
	
	Now, we deal with $I_{7}$.
	\begin{eqnarray}
		I_7&\leq & \bigg\|\int_{s}^{(\lfloor \frac{s}{\delta} \rfloor+1)\delta}(f(\bar X_{r_{\delta}},Y_{F}^\eps(\theta_{r}\omega_2, \bar X_{r_{\delta}}))-\bar f(\bar X_{r_{\delta}}))\,dr\bigg\|\cr
		&&+\bigg\|\sum_{k=\lfloor \frac{s}{\delta} \rfloor+1}^{\lfloor \frac{t}{\delta} \rfloor-1}\int_{k\delta}^{(k+1)\delta}(f(\bar X_{r_{\delta}},Y_{F}^\eps(\theta_{r}\omega_2, \bar X_{r_{\delta}}))-\bar f(\bar X_{r_{\delta}}))\,dr\bigg\|\cr
		&&+\bigg\|\int_{\lfloor \frac{t}{\delta} \rfloor \delta}^{t}(f(\bar X_{r_{\delta}},Y_{F}^\eps(\theta_{r}\omega_2, \bar X_{r_{\delta}}))-\bar f(\bar X_{r_{\delta}}))\,dr\bigg\|\cr
		&\leq & C  (t-s)^{1-\gamma_0}\delta^{\gamma_0}\cr
		&& \quad + \frac{(t-s)}{\delta}\max_{\lfloor \frac{s}{\delta} \rfloor+1 \leq k \leq \lfloor \frac{t}{\delta} \rfloor-1} \bigg\|\int_{k\delta}^{(k+1)\delta}(f(\bar X_{r_{\delta}},Y_{F}^\eps(\theta_{r}\omega_2, \bar X_{r_{\delta}}))-\bar f(\bar X_{r_{\delta}}))\,dr\bigg\|\cr
		&\leq & C  (t-s)^{1-\gamma_0}\delta^{\gamma_0}\cr&&+
		\frac{(t-s)}{\delta}\max_{0 \leq k \leq \lfloor \frac{T}{\delta} \rfloor-1} \bigg\|\int_{0}^{(k+1)\delta}(f(\bar X_{r_{\delta}},Y_{F}^\eps(\theta_{r}\omega_2, \bar X_{r_{\delta}}))-\bar f(\bar X_{r_{\delta}}))\,dr\bigg\|\cr&& +
		\frac{(t-s)}{\delta}\max_{0 \leq k \leq \lfloor \frac{T}{\delta} \rfloor-1} \bigg\|\int_{0}^{k\delta}(f(\bar X_{r_{\delta}},Y_{F}^\eps(\theta_{r}\omega_2, \bar X_{r_{\delta}}))-\bar f(\bar X_{r_{\delta}}))\,dr\bigg\|
	\end{eqnarray}
	where $\gamma_0\in(0,1)$ which will be chosen later and $C$ depends on $\|f\|_{\infty}, \gamma_0$. Since
	\begin{align*}
		&\max_{0\leq k \leq \lfloor \frac{T}{\delta} \rfloor-1} \bigg\|\int_{0}^{(k+1) \delta}(f(\bar X_{k \delta},Y_{F}^\eps(\theta_{r}\omega_2, \bar X_{k \delta}))-\bar f(\bar X_{k \delta}))\,dr\bigg\|\\
		\le & \max_{0\leq k \leq \lfloor \frac{T}{\delta} \rfloor-1}  T \frac{\eps}{(k+1) \delta} \bigg\|\int_{0}^{\frac{(k+1) \delta}{\eps}}(f(\bar X_{k \delta},Y_{F}^1(\theta_{r}\omega_2, \bar X_{k \delta}))-\bar f(\bar X_{k \delta}))\,dr\bigg\|
	\end{align*}
	we have for  $\eps \rightarrow 0 $, $\frac{(k+1) \delta}{\eps} \rightarrow +\infty$ for any $k, 0 \le k\le \lfloor \frac{T}{\delta} \rfloor-1$. In addition we take the maximum over finitely many elements determined by the fixed number $\delta$ given. Following \cref{l3}, we have for every element under the maximum
	\begin{eqnarray}
		\quad   \max_{0\leq k \leq \lfloor \frac{T}{\delta} \rfloor-1} \frac{T\eps}{(k+1) \delta} \bigg\|\int_{0}^{\frac{(k+1) \delta}{\eps}}(f(\bar X_{k \delta},Y_{F}^1(\theta_{r}\omega_2, \bar X_{k \delta}))-\bar f(\bar X_{k \delta}))\,dr\bigg\| \leq C_\eps
	\end{eqnarray}
	where $C_\eps \rightarrow 0,$ as $\eps \rightarrow 0$. We note that by \cref{l3} and Remark \ref{3.11} (3) we can consider as an argument the random variable $\bar X_{k \delta}(\omega_1)$ which is independent of $\omega_2$ inside the integrand of the last integral.	
	Thus, we have for any $\eps<\eps^\prime_0(\delta)$  sufficiently small and the $\delta$ given
	\begin{eqnarray}
		I_7 \le C  (t-s)^{1-\gamma_0}\delta^{\gamma_0}+\frac{C_\eps}{\delta} (t-s)^{1-\gamma_0}T^{\gamma_0}\le C  (t-s)^{1-\gamma_0}\delta^{\gamma_0}.
	\end{eqnarray}
	
	Then, for $t-s>\delta$ and $\eps<\eps^\prime_0(\delta)$, we have
	\begin{align} \label{q}
		\begin{split}
			\|Q^{\eps}_{s,t}-\bar Q_{s,t}\| \leq  C \delta^{\gamma}(t-s)+C  (t-s)^{1-\gamma_0}\delta^{\gamma_0}+C_f \int_{s}^{t} \|X^\eps_r-\bar X_r\|\,dr
		\end{split}
	\end{align}
	where $\gamma_0\in(0,1)$ which will be chosen later, $C$ may depend on $T,\|f\|_{\infty},C_f,C_2,C_5,C_6$ but independent of $\delta, \eps$.

	Next, consider the $p$-variation norm of $Q^{\eps}_{s,t}-\bar Q_{s,t}$ directly, by (\ref{q}) and taking $p(1-\gamma_0)>1$,  for $t-s>\delta$ and $\delta>0$, one has
	\begin{eqnarray}\label{Qp}
		&&\interleave Q^{\eps}-\bar Q \interleave_{p-{\rm var}, [s,t]}\cr
		&= & \bigg(\sup _{\mathcal{P}([s,t])} \sum_{[t_{i}, t_{i+1}] \in \mathcal{P}([s,t])}\|Q^\eps_{t_{i}, t_{i+1}}-\bar Q_{t_{i}, t_{i+1}}\|^{p}\bigg)^{\frac{1}{p}}\cr
		&\leq& C \bigg(\sup _{\mathcal{P}([s,t])} \sum_{[t_{i}, t_{i+1}] \in \mathcal{P}([s,t])}\bigg( \delta^\gamma (t_{i+1}-t_i)+(t_{i+1}-t_i)^{1-\gamma_0}\delta^{\gamma_0}\cr
		&&\quad+\int_{t_{i}}^{t_{i+1}}\|X^\eps_r-\bar X_r\| \, dr\bigg)^{p}\bigg)^{\frac{1}{p}}\cr
		&\leq  &C \bigg(\sup _{\mathcal{P}([s,t])} \sum_{[t_{i}, t_{i+1}] \in \mathcal{P}([s,t])}\big( \delta^\gamma (t_{i+1}-t_i)\big)^{p}\bigg)^{\frac{1}{p}}\cr
		&&+C \bigg(\sup _{\mathcal{P}([s,t])} \sum_{[t_{i}, t_{i+1}] \in \mathcal{P}([s,t])}\big( (t_{i+1}-t_i)^{1-\gamma_0}\delta^{\gamma_0} \big)^{p}\bigg)^{\frac{1}{p}}\cr
		&&+ C_f \bigg(\sup _{\mathcal{P}([s,t])} \sum_{[t_{i}, t_{i+1}] \in \mathcal{P}([s,t])}\bigg( \int_{t_{i}}^{t_{i+1}}\|X^\eps_r-\bar X_r\| \, dr\bigg)^{p}\bigg)^{\frac{1}{p}}\cr
		&\leq  &C \delta^\gamma \bigg(\sup _{\mathcal{P}([s,t])} \bigg(\sum_{[t_{i}, t_{i+1}] \in \mathcal{P}([s,t])} (t_{i+1}-t_i)\bigg)^{p}\bigg)^{\frac{1}{p}}\cr
		&& +C \delta^{\gamma_0}\bigg(\sup _{\mathcal{P}([s,t])} \bigg(\sum_{[t_{i}, t_{i+1}] \in \mathcal{P}([s,t])} (t_{i+1}-t_i)\bigg)^{(1-\gamma_0)p} \bigg)^{\frac{1}{p}}\cr
		&&+ C_f \bigg(\sup _{\mathcal{P}([s,t])} \bigg(\sum_{[t_{i}, t_{i+1}] \in \mathcal{P}([s,t])} \int_{t_{i}}^{t_{i+1}}\|X^\eps_r-\bar X_r\| \, dr\bigg)^{p}\bigg)^{\frac{1}{p}}\cr
		&\leq & C \delta+ C \delta^{\gamma_0}+C_f \int_{s}^{t}\|X^\eps_r-\bar X_r\| \, dr.
	\end{eqnarray}
	Thus, by (\ref{q0}) and (\ref{Qp}), for any $s,t\in[0,T]$ and $\delta >0$, we have
	\begin{align*}
		\begin{split}
			\interleave Q^{\eps}-\bar Q \interleave_{p-{\rm var}, [s,t]} \leq   C (\delta+  \delta^{\gamma_0})+C_f \int_{s}^{t}\|X^\eps_r-\bar X_r\| \, dr.
		\end{split}
	\end{align*}
The proof is completed.	
\end{proof}
\begin{theorem}\label{mainthm}
	Let {\rm (A1)-(A5)} hold. For any $X_0\in \mathbb{R}^n, Y_0\in \mathbb{R}^m$, as $\epsilon\rightarrow0$ the solution of {\rm (\ref{slowpath})} converges to $\bar X$, that is, we have
	\begin{align}\label{main}
		\|X^\eps-\bar{X}\|_{\infty}+\interleave X^\eps-\bar{X}\interleave_{p-{\rm var},[0, T]}+\interleave R^{X^\eps}-R^{\bar{X}}\interleave_{q-{\rm var},[0, T]^{2}} \rightarrow 0
	\end{align}
	where $q=\frac{p}{2}$.
\end{theorem}
\begin{proof}
	
	In the following proof a constant $C$ appears which can change from inequality to  inequality.  $C$ may depend on $T, C_f,C_h,\|f\|_{\infty},p,\|X_0\|,\|Y_0\|,\interleave \boldsymbol{\omega_1}\interleave_{p-{\rm var},[0, T]}$ and $\interleave \omega_2\interleave_{\gamma,[0,T]}$
	but independent of $\delta, \eps$.
	We have to show that on some subset of $\Omega$ of full measure, there exists an $\eps_0(\delta)>0$ so that for
	$\eps<\eps_0(\delta)$ we have that the left side of (\ref{main}) is bound. Fixing a $\delta$-partition of $[0,T]$ we have
	\begin{align*}
		\|X^\eps-\bar{X}\|_{\infty}+\interleave X^\eps-\bar{X}\interleave_{p-{\rm var},[0, T]}+\interleave R^{X^\eps}-R^{\bar{X}}\interleave_{q-{\rm var},[0, T]^{2}} \leq
		2C(\delta^\gamma+\delta^{\gamma_0}),
	\end{align*}
	where $\delta>0$ and the last inequalities hold when $\eps<\eps_0(\delta)$. 
	
	Now, we begin to estimate (\ref{main}). For any $s,t\in[0,T] $, let
	\begin{eqnarray*}
	Z^\eps_{t}&=&\int_{0}^{t}h(X^\eps_r)\,d\boldsymbol{\omega_1}, \quad \bar {Z}_{t}=\int_{0}^{t}h(\bar X_r)\,d\boldsymbol{\omega_1}.
	\end{eqnarray*}
	
	Then, by (\ref{slowpath}) and  (\ref{x-ave}),  one has
	\begin{eqnarray*}
		\|X^\eps_{s,t}-\bar{X}_{s,t}\|
		\leq \|Q^\eps_{s,t}-\bar Q_{s,t}\|+\|Z^\eps_{s,t}-\bar Z_{s,t}\|.
	\end{eqnarray*}

	For $\|Z^\eps_{s,t}-
	\bar{Z}_{s,t}\|$, by triangle inequality and constructing the form like (\ref{inq-1})  such that
	\begin{eqnarray}\label{zzz}
		\|Z^\eps_{s,t}-
		\bar{Z}_{s,t}\|& \leq&\big\|Z^\eps_{s,t}-
		\bar{Z}_{s,t}-(h(X^\eps_s)-h(\bar X_s))\omega_{1,s,t}-[h(X^\eps)-h(\bar{X})]_s^\prime \bbomega_{1,s,t}\big\|\cr
		&&+
		\|(h(X^\eps_s)-h(\bar{X}_s)) \omega_{1,s,t}\|+\|[h(X^\eps)-h(\bar{X})]_s^\prime \bbomega_{1,s,t}\| \cr
		&=: &J_1+J_2+J_3
	\end{eqnarray}
	where $[h(X^\eps)-h(\bar{X})]^\prime_s=Dh(X^\eps_s)h(X^\eps_s)-Dh(\bar{X}_s)h(\bar{X}_s)$.

	For $J_1$, by (\ref{inq-1}), we have
	\begin{align}\label{Qz}
		\begin{split}
			J_1
			\le &
			C_p  \big(\interleave \omega_1\interleave_{p-{\rm var}, [s,t]}\interleave R^{h(X^\eps)}-R^{h(\bar X)}\interleave_{q-{\rm var},[s, t]^{2}}\\
			&\quad\quad+ \interleave[h(X^\eps)-h(\bar{X})]^\prime\interleave_{p-{\rm var}, [s,t]}\interleave \bbomega_1 \interleave_{q-{\rm var}, [s,t]^{2}} \big)
		\end{split}
	\end{align}
	where $C_p>1$.

	To estimate $\interleave R^{h(X^\eps)}-R^{h(\bar X)}\interleave_{q-{\rm var},[s, t]^{2}}$, by (\ref{rh}) and
	\begin{align}\label{xx}
		X^\eps_{s,t}=h(X^\eps_s) \omega_{1,s,t}+R^{X^\eps}_{s,t},\quad
		\bar{X}_{s,t}=h(\bar{X}_s) \omega_{1,s,t}+R^{\bar{X}}_{s,t},
	\end{align}
	we have $$
	R_{s, t}^{h(X^\eps)}=h(X^\eps)_{s, t}-D h(X^\eps_s)X^\eps_{s,t}+ D h(X^\eps_s) R^{X^\eps}_{s,t}$$
	and
	taking the difference with $R_{s, t}^{h(\bar X)}$ where we replace $X^\eps, D h(X^\eps), R^{X^\eps}$ above by $\bar X,D h(\bar X), R^{\bar X}$, respectively, leads to the bound \cite[p.
	110]{friz2020course}
	\begin{eqnarray}
		&&\| R_{s, t}^{h(X^\eps)}-R_{s, t}^{h(\bar{X})}\|\cr
		&\leq & \|h(X^\eps)_{s, t}-D h(X^\eps_s)X^\eps_{s,t}-h(\bar X)_{s, t}+D h(\bar X_s) \bar X_{s,t}\|\cr
		&&+\|D h(X^\eps_s) R^{X^\eps}_{s,t}- D h(\bar X_s) R^{\bar X}_{s,t}\|\cr
		&\leq & \int^{1}_0 \big(D^2h(X^\eps_s+\theta X^\eps_{s,t})( X^\eps_{s,t}, X^\eps_{s,t})\cr
		&&\quad-D^2h(\bar X_s+\theta \bar X_{s,t})( \bar X_{s,t}, \bar X_{s,t})\big) (1-\theta)d\theta\cr
		&&+C_h(\|X^\eps_s-\bar X_s\|\|R^{X^\eps}_{s,t}\|+\|R^{X^\eps}_{s,t}-R^{\bar X}_{s,t}\|)\cr
		&\leq & C_h\big(3\|X^\eps-\bar X\|_{\infty,[s,t]}\|X^\eps_{s,t}\|^2+\|X^\eps_{s,t}\|\|X^\eps_{s,t}-\bar X_{s,t}\|\cr
		&&+\|X^\eps_{s,t}-\bar X_{s,t}\|\|\bar X_{s,t}\|\big)+C_h(\|X^\eps-\bar X\|_{\infty,[s,t]}\|R^{X^\eps}_{s,t}\|+\|R^{X^\eps}_{s,t}-R^{\bar X}_{s,t}\|)\cr
		&\leq &3 C_h \|X^\eps-\bar X\|_{\infty,[s,t]}\interleave  X^\eps\interleave_{p-{\rm var},[s,t]}^2\cr
		&&+2C_h \interleave X^\eps -\bar X\interleave_{p-{\rm var},[s, t]} (\interleave X^\eps\interleave_{p-{\rm var},[s, t]}+\interleave \bar{X}\interleave_{p-{\rm var},[s, t]})\cr
		&&+C_h(\|X^\eps-\bar X\|_{\infty,[s,t]} \interleave R^{X^\eps}\interleave_{q-{\rm var},[s, t]^2}+ \interleave R^{X^\eps}-R^{\bar X}\interleave_{q-{\rm var},[s,t]^2}).
	\end{eqnarray}
		By (\ref{controlex}), it is not difficulty to verify that
		\begin{align*}
			\begin{split}
				\|X^\eps-\bar X\|_{\infty,[s,t]}^{\frac p2},
				\interleave  X^\eps\interleave_{p-{\rm var},[s,t]}^{p}, \big(\interleave X^\eps -\bar X\interleave_{p-{\rm var},[s, t]} \cdot \interleave X^\eps\interleave_{p-{\rm var},[s, t]})\big)^{\frac p2},\\
				\big(\interleave X^\eps -\bar X\interleave_{p-{\rm var},[s, t]}  \cdot \interleave \bar{X}\interleave_{p-{\rm var},[s, t]})\big)^{\frac p2}, \interleave R^{X^\eps}\interleave_{{\frac p2}-{\rm var},[s, t]^2}^{\frac p2}, \interleave R^{X^\eps}-R^{\bar X}\interleave_{{\frac p2}-{\rm var},[s,t]^2}^{\frac p2}
			\end{split}
		\end{align*}
		are all control functions.
	Hence, by Lemma \ref{control} (ii), one has
	\begin{align}\label{rhx}
		\begin{split}
			\interleave & R^{h(X^\eps)}-R^{h(\bar{X})}\interleave_{{\frac p2}-{\rm var},[s,t]^2}\\
			\leq & 3 C_h L_{5,q}^{1/q}\|X^\eps-\bar X\|_{\infty,[s,t]}\interleave  X^\eps\interleave_{p-{\rm var},[s,t]}^2\\
			&+2C_h L_{5,q}^{1/q}\interleave X^\eps -\bar X\interleave_{p-{\rm var},[s, t]} (\interleave X^\eps\interleave_{p-{\rm var},[s, t]}+\interleave \bar{X}\interleave_{p-{\rm var},[s, t]})\\
			&+C_hL_{5,q}^{1/q}(\|X^\eps-\bar X\|_{\infty,[s,t]} \interleave R^{X^\eps}\interleave_{{\frac p2}-{\rm var},[s, t]^2}+ \interleave R^{X^\eps}-R^{\bar X}\interleave_{{\frac p2}-{\rm var},[s,t]^2}).
		\end{split}
	\end{align}
	where $L_{5,q}^{1/q} \geq 1$ is defined in Lemma \ref{control} for $n=5$.
	
	Then, we deal with $\interleave[h(X^\eps)-h(\bar{X})]^\prime\interleave_{p-{\rm var}, [s,t]}$. Because the following inequality
	\begin{align}\label{inq-2}
		\begin{split}
			\|\sigma(u_{1})-\sigma(v_{1})-\sigma(u_{2})+\sigma(v_{2})\| \leq & C_{\sigma}\|u_{1}-v_{1}-u_{2}+v_{2}\|\\
			&+C_{\sigma}\|u_{1}-u_{2}\|(\|u_{1}-v_{1}\|+\|u_{2}-v_{2}\|)
		\end{split}
	\end{align}
	holds where $\sigma$ is differentiable \cite[Lemma 7.1]{Nualart2002} and replacing $\sigma(\cdot)$ by $Dh(\cdot)h(\cdot)$, we have
	\begin{align}\label{9212}
		\begin{split}
			\|(D h(X^\eps) h(X^\eps))_{s, t}&-(D h(\bar{X})h(\bar{X}))_{s, t}\|\\
			\leq & 2C^2_{h}\big(\|X^\eps_{s, t}-\bar{X}_{s, t}\|+\|X^\eps_t-\bar{X}_t\|(\|X^\eps_{s, t}\|+\|\bar{X}_{s, t}\|)\big) \\
			\leq &2C^2_{h}\big(\interleave X^\eps-\bar{X}\interleave_{p-{\rm var},[s, t]}\\
			&+\|X^\eps-\bar{X}\|_{\infty,[s, t]}(\interleave X^\eps \interleave_{p-{\rm var},[s, t]}+\interleave \bar{X} \interleave_{p-{\rm var},[s, t]})\big).
		\end{split}
	\end{align}
	
	Because $\interleave X^\eps-\bar{X}\interleave_{p-{\rm var},[s, t]}^p$, $\|X^\eps-\bar{X}\|_{\infty,[s, t]}^p$, $\interleave X^\eps \interleave_{p-{\rm var},[s, t]}^p$ and $\interleave \bar{X} \interleave_{p-{\rm var},[s, t]}^p$ are all control functions, then by Lemma \ref{control} (i), one has
	\begin{align}\label{hh}
		\begin{split}
			\interleave [h(X^\eps)&-h(\bar{X})]^\prime\interleave_{p-{\rm var}, [s,t]}\\
			\leq & 2C^2_{h}L_{3,p}^{1/p}\big(\interleave X^\eps-\bar{X}\interleave_{p-{\rm var},[s, t]}\\
			&+\|X^\eps-\bar{X}\|_{\infty,[s, t]}(\interleave X^\eps \interleave_{p-{\rm var},[s, t]}+\interleave \bar{X} \interleave_{p-{\rm var},[s, t]})\big).
		\end{split}
	\end{align}

	Then, by (\ref{Qz}),(\ref{rhx}) and (\ref{hh}), we have
	\begin{eqnarray}\label{I93}
		J_1
		&\leq &
		C_p  L_{5,q}^{1/q}\interleave \omega_1\interleave_{p-{\rm var}, [s,t]}\big(3 C_h \|X^\eps-\bar X\|_{\infty,[s,t]}\interleave  X^\eps\interleave_{p-{\rm var},[s,t]}^2\cr
		&&+2C_h \interleave X^\eps -\bar X\interleave_{p-{\rm var},[s, t]} (\interleave X^\eps\interleave_{p-{\rm var},[s, t]}+\interleave \bar{X}\interleave_{p-{\rm var},[s, t]})\cr
		&&+C_h(\|X^\eps-\bar X\|_{\infty,[s,t]} \interleave R^{X^\eps}\interleave_{q-{\rm var},[s, t]^2}+ \interleave R^{X^\eps}-R^{\bar X}\interleave_{q-{\rm var},[s,t]^2})\big)\cr
		&&+ 2 C_p L_{3,p}^{1/p}C^2_{h} \interleave \bbomega_1 \interleave_{q-{\rm var}, [s,t]^{2}}\big(\interleave X^\eps-\bar{X}\interleave_{p-{\rm var},[s, t]}\cr
		&&+\|X^\eps-\bar{X}\|_{\infty,[s, t]}(\interleave X^\eps \interleave_{p-{\rm var},[s, t]}+\interleave \bar{X} \interleave_{p-{\rm var},[s, t]})\big)=:\Theta
	\end{eqnarray}
	where $L_{3,p}^{1/p} \geq 1$ is defined in Lemma \ref{control} for $n=3$ and
	\begin{eqnarray*}
		\Theta
		&=&
		3 C_p  L_{5,q}^{1/q} C_h \interleave \omega_1\interleave_{p-{\rm var}, [s,t]} \|X^\eps-\bar X\|_{\infty,[s,t]}\interleave  X^\eps\interleave_{p-{\rm var},[s,t]}^2\cr
		&&+2C_p L_{5,q}^{1/q} C_h\interleave \omega_1\interleave_{p-{\rm var}, [s,t]} \interleave X^\eps -\bar X\interleave_{p-{\rm var},[s, t]}(\interleave X^\eps\interleave_{p-{\rm var},[s, t]}+\interleave \bar{X}\interleave_{p-{\rm var},[s, t]})\cr
		&&+C_p  L_{5,q}^{1/q} C_h \interleave \omega_1\interleave_{p-{\rm var}, [s,t]}\|X^\eps-\bar X\|_{\infty,[s,t]} \interleave R^{X^\eps}\interleave_{q-{\rm var},[s, t]^2}\cr
		&&+C_p  L_{5,q}^{1/q} C_h \interleave \omega_1\interleave_{p-{\rm var}, [s,t]}\interleave R^{X^\eps}-R^{\bar X}\interleave_{q-{\rm var},[s,t]^2}\cr
		&&+ 2 C_p L_{3,p}^{1/p}C^2_{h} \interleave \bbomega_1 \interleave_{q-{\rm var}, [s,t]^{2}}\interleave X^\eps-\bar{X}\interleave_{p-{\rm var},[s, t]}\cr
		&&+ 2 C_p L_{3,p}^{1/p} C^2_{h} \interleave \bbomega_1 \interleave_{q-{\rm var}, [s,t]^{2}}\|X^\eps-\bar{X}\|_{\infty,[s, t]}(\interleave X^\eps \interleave_{p-{\rm var},[s, t]}+\interleave \bar{X} \interleave_{p-{\rm var},[s, t]}).
	\end{eqnarray*}
	
	For $J_2+J_3$, by (A2), we have
	\begin{align}\label{91}
		\begin{split}
			J_2+J_3 \leq & C_{h}\|X^\eps-\bar{X}\|_{\infty,[s, t]}\interleave \omega_1 \interleave_{p-{\rm var},[s, t]}\\
			&+2C_h^2 \|X^\eps-\bar{X}\|_{\infty,[s, t]} \interleave \bbomega_1\interleave_{q-{\rm var},[s, t]^{2}}.
		\end{split}
	\end{align}
	
	Then, by (\ref{91}) and (\ref{I93}), we have
	\begin{eqnarray}\label{z-z}
		\|Z^\eps_{s,t}-
		\bar{Z}_{s,t}\| &\leq & J_1+J_2+J_3\cr
		&\leq & C_{h}\interleave \omega_1 \interleave_{p-{\rm var},[s, t]}\|X^\eps-\bar{X}\|_{\infty,[s, t]}\cr
		&&+
		2C_h^2 \interleave \bbomega_1\interleave_{q-{\rm var},[s, t]^{2}} \|X^\eps-\bar{X}\|_{\infty,[s, t]} +\Theta.
	\end{eqnarray}

	Denote
	\begin{align*}
		\Psi^{p,q}_{\infty,[s,t]}(X^\eps,\bar X,  R^{X^\eps}):=&\big(1+\interleave X^\eps\interleave_{p-{\rm var},[s, t]}+\interleave  X^\eps\interleave_{p-{\rm var},[s, t]}^2\\
		&+\interleave \bar{X}\interleave_{p-{\rm var},[s,t]}+ \interleave R^{X^\eps}\interleave_{q-{\rm var},[s, t]^2}\big)\\
		d_{\infty,[s,t]}^{p,q}(X^\eps,\bar X;R^{X^\eps},R^{\bar X} )
		:=&\big(\|X^\eps-\bar X\|_{\infty,[s,t]}+\interleave X^\eps -\bar X\interleave_{p-{\rm var},[s, t]}\\
		&+
		\interleave R^{X^\eps}-R^{\bar X}\interleave_{q-{\rm var},[s,t]^2}\big)
	\end{align*}
	and
	note that
	\begin{align*}
		\interleave X^\eps\interleave_{p-{\rm var},[s, t]}^p,  \interleave  X^\eps\interleave_{p-{\rm var},[s,t]}^{2p},\interleave \bar{X}\interleave_{p-{\rm var},[s, t]}^p,  \interleave R^{X^\eps}\interleave_{q-{\rm var},[s, t]^2},\\ \|X^\eps-\bar{X}\|_{\infty,[s, t]}^p,  \interleave X^\eps-\bar{X}\interleave_{p-{\rm var},[s, t]}^p, \interleave R^{X^\eps}-R^{\bar X}\interleave_{q-{\rm var},[s,t]^2}^p
	\end{align*}
	are all control functions and $C_p>1$, then by Lemma \ref{control} (i), for any $s,t\in[0,T]$, one has
	\begin{align}
		\begin{split}
			\interleave Z^\eps-&\bar{Z}\interleave_{p-{\rm var},[s, t]}\\
			\leq & 14 C_p L_{p,q} (C _{h}^{2} \interleave \boldsymbol{\omega_1}\interleave_{p-{\rm var},[s, t]}^{2} \vee C_{h}\interleave \boldsymbol{\omega_1}\interleave_{p-{\rm var},[s, t]}) \\
			&\quad  \times\Psi^{p,q}_{\infty,[s,t]}(X^\eps,\bar X,  R^{X^\eps})d_{\infty,[s,t]}^{p,q}(X^\eps,\bar X;R^{X^\eps},R^{\bar X} )
		\end{split}
	\end{align}
	where $L_{p,q}\geq 1$ may depends on $L^{1/p}_{n,p}, L^{1/q}_{n,q}$ for some $n$, we also use the fact, see (\ref{xxx}), that
	\begin{align}\label{chh}
		\begin{split}
			\interleave \omega_1 \interleave_{p-{\rm var},[s, t]} \leq \interleave \boldsymbol{\omega_1}\interleave_{p-{\rm var},[s, t]} \, \, \interleave \bbomega_1 \interleave_{q-{\rm var}, [s,t]^{2}}\leq \interleave \boldsymbol{\omega_1}\interleave_{p-{\rm var},[s, t]}^{2}.
		\end{split}
	\end{align}
	
	Then, by the estimate in Lemma \ref{con}, for any $s,t\in[0,T]$ and $\delta>0$, we have
	\begin{align}\label{xpvar}
		\begin{split}
			\interleave X^\eps-\bar{X}\interleave_{p-{\rm var},[s, t]}
			\leq&  C_f \int_{s}^{t} \|X^\eps_r-\bar X_r\|\,dr+C (\delta^\gamma + \delta^{\gamma_0})\\
			&+ 14 C_pL_{p,q} (C _{h}^{2} \interleave \boldsymbol{\omega_1}\interleave_{p-{\rm var},[s, t]}^{2} \vee C_{h}\interleave \boldsymbol{\omega_1}\interleave_{p-{\rm var},[s, t]}) \\
			&\quad  \times\Psi^{p,q}_{\infty,[s,t]}(X^\eps,\bar X,  R^{X^\eps})d_{\infty,[s,t]}^{p,q}(X^\eps,\bar X;R^{X^\eps},R^{\bar X} )
		\end{split}
	\end{align}
	where $\gamma_0<1-1/p$, $C$ may depend on $T,\|f\|_{\infty},C_f,C_2,C_5,C_6$ but independent of $\delta, \eps$.
	
	Now, we estimate $\| X^\eps-\bar{X}\|_{\infty,[s, t]} $ by the fact that
	\begin{align}\label{xinf} \| X^\eps-\bar{X}\|_{\infty,[s, t]} \le \| X^\eps_s-\bar{X}_s\|+
		\interleave X^\eps-\bar{X}\interleave_{p-{\rm var},[s, t]}.
	\end{align}
	
	Next, for $\interleave R^{X^\eps}-R^{\bar{X}}\interleave_{q-{\rm var},[s, t]^{2}}$, by (\ref{xx}), we begin with
	\begin{align}\label{RRR}
		\begin{split}
			\|R_{s,t}^{X^\eps}&-R_{s,t}^{\bar X}\| \\
			= & \|X_{s,t}^{\eps}-\bar X_{s,t}-(h(X^\eps_s)-h(\bar X_s))\omega_{1,s,t}\|
			\\
			\leq& \|Q_{s,t}^\eps- \bar Q_{s,t}\|+\|[h(X^\eps)-h(\bar{X})]_s^\prime \bbomega_{1,s,t}\|\\
			&+\|Z_{s,t}^\eps- \bar Z_{s,t}-(h(X^\eps_s)-h(\bar X_s))\omega_{1,s,t}-[h(X^\eps)-h(\bar{X})]_s^\prime \bbomega_{1,s,t}\|
			\\
			= & \|Q_{s,t}^\eps- \bar Q_{s,t}\|+J_1+J_3.
		\end{split}
	\end{align}
	
	Since $q=p/2$, in what follows, we can chose a new $\gamma_0$ such that $\gamma_0<1-1/q$ which implies that $\gamma_0<1-1/p$ still holds.
	Then, similar to the estimate of Lemma \ref{con}, for any $s,t\in[0,T]$ and $\delta>0$, one has
	\begin{align}\label{Qq}
		\|Q_{s,t}^\eps- \bar Q_{s,t}\| \leq \interleave Q^\eps- \bar Q\interleave_{q-{\rm var},[s, t]}
		\leq  C (\delta+ \delta^{\gamma_0})+C_f \int_{s}^{t}\|X^\eps_r-\bar X_r\| \, dr
	\end{align}
	where $\gamma_0<1-1/q$, $C$ may depend on $T,\|f\|_{\infty},C_f,C_2,C_5,C_6$ but independent of $\delta, \eps$.
	
	Next, using Lemma \ref{control} (ii) and (\ref{chh}) to obtain the $q-$variation norm of the terms $J_2+J_3$
	and by (\ref{chh}), together with (\ref{Qq}), for any $s,t\in[0,T]$ and $\delta>0$, we obtain
	\begin{align*}
		\interleave R^{X^\eps}-R^{\bar{X}}\interleave_{q-{\rm var},[s, t]^{2}}
		\leq & C_f \int_{s}^{t} \|X^\eps_r-\bar X_r\|\,dr+C (\delta^\gamma+\delta^{\gamma_0})\\
		&+ 13 C_p L_{p,q} \big(C _{h}^{2} \interleave \boldsymbol{\omega_1}\interleave_{p-{\rm var},[s, t]}^{2} \vee C_{h}\interleave \boldsymbol{\omega_1}\interleave_{p-{\rm var},[s, t]}\big) \\
		& \quad \times\Psi^{p,q}_{\infty,[s,t]}(X^\eps,\bar X,  R^{X^\eps})d_{\infty,[s,t]}^{p,q}(X^\eps,\bar X; R^{X^\eps},R^{\bar X}\big).
	\end{align*}
	
	Furthermore, choosing  $\gamma_0<1-1/q$,  for any $s,t\in[0,T]$ and $\delta>0$, we have
	\begin{align*}
		d_{\infty,[s,t]}^{p,q}(X^\eps,\bar X; R^{X^\eps},R^{\bar X}\big)
		\leq & 3 C_f \int_{s}^{t} \|X^\eps_r-\bar X_r\|\,dr+ C (\delta^\gamma+\delta^{\gamma_0})+\|X^\eps_s-\bar X_s\|\\
		&+ 41 C_p L_{p,q} \big(C _{h}^{2} \interleave \boldsymbol{\omega_1}\interleave_{p-{\rm var},[s, t]}^{2} \vee C_{h}\interleave \boldsymbol{\omega_1}\interleave_{p-{\rm var},[s, t]}\big) \\
		& \quad \times\Psi^{p,q}_{\infty,[s,t]}(X^\eps,\bar X,  R^{X^\eps})d_{\infty,[s,t]}^{p,q}(X^\eps,\bar X; R^{X^\eps},R^{\bar X}\big)
	\end{align*}
	where $C$ may depend on $T,\|f\|_{\infty},C_f,C_2,C_5,C_6$ but independent of $\delta, \eps$.
	
	Now, we choose \begin{align*}
		\tilde{\nu}=\big(82 C_p L_{p,q} C_h\Psi^{p,q}_{\infty,[0,T]}(X^\eps,\bar X,  R^{X^\eps})\big)^{-1}.
	\end{align*}  and
	let $$
	\tilde{\tau}_{0}=0, \quad \tilde{\tau}_{i+1}:=\inf \left\{t>\tilde{\tau}_{i}:\interleave \boldsymbol{\omega_1}\interleave_{p-{\rm var},\left[\tilde{\tau}_{i}, t\right]}=\tilde{\nu}\right\} \wedge T
	$$
	such that
	\begin{align*}41 C_p L_{p,q} C _{h}& \interleave \boldsymbol{\omega_1}\interleave_{p-{\rm var},[\tilde{\tau}_{i}, \tilde{\tau}_{i+1}]}\Psi^{p,q}_{\infty,[0,T]}(X^\eps,\bar X,  R^{X^\eps})\leq \frac{1}{2}.
	\end{align*}
	
	Since  $C_p>1$ and $\Psi^{p,q}_{\infty,[0,T]}(X^\eps,\bar X,  R^{X^\eps})\geq 1$, for any  $s,t\in [\tilde{\tau}_{i}, \tilde{\tau}_{i+1}]$,  we obtain that $C _{h}^{2} \interleave \boldsymbol{\omega_1}\interleave_{p-{\rm var},[s, t]}^{2}\leq C _{h} \interleave \boldsymbol{\omega_1}\interleave_{p-{\rm var},[s, t]}.$  Then, for any $s,t\in[0,T]$ and $\delta>0$, we have
	\begin{align*}
		d_{\infty,[s,t]}^{p,q}(X^\eps,\bar X;R^{X^\eps},R^{\bar X} )
		\leq&  6 C_f \int_{s}^{t}d_{\infty,[s,r]}^{p,q}(X^\eps,\bar X;R^{X^\eps},R^{\bar X} )\,dr\\
		&+ (2+6 C_f(t-s))\|X^\eps_s-\bar X_s\|+C (\delta^\gamma+\delta^{\gamma_0}).
	\end{align*}
	
	Applying the Gronwall-lemma \cite[Lemma 6.1, p
	89]{amann2011ordinary},  for any $t\in[\tilde{\tau}_{i},\tilde{\tau}_{i+1}]$ and $\delta>0$, we have
	\begin{eqnarray}\label{dinfty}
		&&d^{p,q}_{\infty,[\tilde{\tau}_{i}, t]}(X^\eps,\bar X;R^{X^\eps},R^{\bar X} )\cr
		&\leq & (2+6C_f(\tilde{\tau}_{i+1}-\tilde{\tau}_{i}))\| X^\eps_{\tilde{\tau}_{i}}-\bar{X}_{\tilde{\tau}_{i}}\|+C (\delta^\gamma+\delta^{\gamma_0})\cr
		&&+6 C_f \int_{\tilde{\tau}_{i}}^{t} e^{6 C_f(t-r)}\big((2+6C_f(\tilde{\tau}_{i+1}-\tilde{\tau}_{i}))\| X^\eps_{\tilde{\tau}_{i}}-\bar{X}_{\tilde{\tau}_{i}}\|+C (\delta^\gamma+\delta^{\gamma_0})\big)\,dr\cr
		&\leq & \big((2+6C_f(\tilde{\tau}_{i+1}-\tilde{\tau}_{i}))\| X^\eps_{\tilde{\tau}_{i}}-\bar{X}_{\tilde{\tau}_{i}}\|+C (\delta^\gamma+\delta^{\gamma_0})\big)\cr
		&&\quad \times\Big(1+
		6 C_f \int_{\tilde{\tau}_{i}}^{t} e^{6 C_f(t-r)}\,dr\Big) \cr
		&\leq & \big(C_{f,T}\| X^\eps_{\tilde{\tau}_{i}}-\bar{X}_{\tilde{\tau}_{i}}\|+C (\delta^\gamma+\delta^{\gamma_0}) \big)e^{6 C_{f}(t-\tilde{\tau}_{i})}
	\end{eqnarray}
	where $C_{f,T}>0$ depends on $C_f, T$ and appears often in the following proof and it may change form line to line.
	
	We now in position to consider $\| X^\eps -\bar{X}\|_{\infty}$.
	\begin{align*}
		\| X^\eps_{\tilde{\tau}_{i+1}}&-\bar{X}_{\tilde{\tau}_{i+1}}\|
		\leq  \| X^\eps-\bar{X}\|_{\infty,[\tilde{\tau}_{i}, \tilde{\tau}_{i+1}]}\\
		\leq& \big(C_{f,T}\| X^\eps_{\tilde{\tau}_{i}}-\bar{X}_{\tilde{\tau}_{i}}\|+C (\delta^\gamma+\delta^{\gamma_0})\big)e^{6 C_f(\tilde{\tau}_{i+1}-\tilde{\tau}_{i})}\\
		\leq& C_{f,T}\| X^\eps_{\tilde{\tau}_{i}}-\bar{X}_{\tilde{\tau}_{i}}\|e^{6 C_f(\tilde{\tau}_{i+1}-\tilde{\tau}_{i})}+C (\delta^\gamma+\delta^{\gamma_0})\\
		\leq& C_{f,T}\| X^\eps-\bar{X}\|_{\infty,[\tilde{\tau}_{i-1}, \tilde{\tau}_{i}]}e^{6 C_f(\tilde{\tau}_{i+1}-\tilde{\tau}_{i})}+C (\delta^\gamma+\delta^{\gamma_0}) \\
		\leq& C_{f,T}(C_{f,T}\| X^\eps_{\tilde{\tau}_{i-1}}-\bar{X}_{\tilde{\tau}_{i-1}}\|e^{6 C_f(\tilde{\tau}_{i}-\tilde{\tau}_{i-1})}+C (\delta^\gamma+\delta^{\gamma_0}))e^{6 C_f(\tilde{\tau}_{i+1}-\tilde{\tau}_{i})}\\
		&+C (\delta^\gamma+\delta^{\gamma_0})\\
		\leq& C_{f,T}\| X^\eps_{\tilde{\tau}_{i-1}}-\bar{X}_{\tilde{\tau}_{i-1}}\|e^{6 C_f(\tilde{\tau}_{i+1}-\tilde{\tau}_{i-1})}+C (\delta^\gamma+\delta^{\gamma_0})
	\end{align*}
	where $C$ may depend on $T,\|f\|_{\infty},C_f,C_2,C_5,C_6$ but independent of $\delta, \eps$.
	
	As a result, for any $\delta>0$, we obtain
	\begin{align}\label{fifi}
		\| X^\eps_{\tilde{\tau}_{i+1}}-\bar{X}_{\tilde{\tau}_{i+1}}\|
		\leq C_{f,T,\tilde{N}}\| X^\eps_{0}-\bar{X}_{0}\|e^{6 C_f T}+C (\delta^\gamma+\delta^{\gamma_0})\leq  C (\delta^\gamma+\delta^{\gamma_0})
	\end{align}
	where $X^\eps_{0}=\bar{X}_{0}$ and $C_{f,T,\tilde{N}}>0$ depends on $C_f, T,\tilde{N}$ and may change form line to line. $\tilde{N}$ is the number of stopping times $\{\tilde{\tau}_i\}_{i\in \mathbb{N}}$ and by Lemma \ref{var}
	\begin{align}\label{NN}
		\tilde{N}\le \tilde{\nu}^{-p}  \interleave \boldsymbol{\omega_1}\interleave^p_{p-{\rm var},[0, T]}+1.
	\end{align}

	Then, for any $\delta>0$, by  (\ref{dinfty}) and (\ref{fifi}),  we show
	\begin{align}\label{ddd}
		\begin{split}
			d^{p,q}_{\infty,[\tilde{\tau}_{i}, \tilde{\tau}_{i+1}]}(X^\eps,\bar X;R^{X^\eps},R^{\bar X} )
			\leq & \big(C_{f,T}\| X^\eps_{\tilde{\tau}_{i}}-\bar{X}_{\tilde{\tau}_{i}}\|+C (\delta^\gamma+\delta^{\gamma_0}) \big)e^{6 C_{f}(\tilde{\tau}_{i+1}-\tilde{\tau}_{i})}\\
			\leq & C_{f,T,\tilde{N}}\| X^\eps_{0}-\bar{X}_{0}\|+C (\delta^\gamma+\delta^{\gamma_0})\leq C (\delta^\gamma+\delta^{\gamma_0})
		\end{split}
	\end{align}
	where $C$ may depend on $T,\|f\|_{\infty},C_f,C_2,C_5,C_6$ and $\tilde{N}$, but independent of $\delta, \eps$.

	Now, for any $\delta>0$ by (\ref{NN}) and (\ref{ddd}), we have
	\[\| X^\eps -\bar{X}\|_{\infty,[0,T]}\leq \sum_{i=0}^{\tilde{N}-1}\| X^\eps-\bar{X}\|_{\infty, [\tilde{\tau}_{i}, \tilde{\tau}_{i+1}]} \leq C (\delta^\gamma+\delta^{\gamma_0})\]
	and  together with Lemma \ref{var}, we obtain
	\begin{align*}
		\begin{split}
			\interleave X^\eps-\bar{X}\interleave_{p-{\rm var},[0, T]} &+\interleave R^{X^\eps}-R^{\bar{X}}\interleave_{q-{\rm var},[0, T]^{2}}\\
			\leq& (\tilde{N}-1)^{\frac{p-1}{p}}\sum_{i=0}^{\tilde{N}-1}\interleave X^\eps-\bar{X}\interleave_{p-{\rm var},[\tilde{\tau}_{i}, \tilde{\tau}_{i+1}]}\\&+
			(\tilde{N}-1)^{\frac{q-1}{q}}\sum_{i=0}^{\tilde{N}-1}\interleave R^\eps-\bar{R}\interleave_{q-{\rm var},[\tilde{\tau}_{i}, \tilde{\tau}_{i+1}]^2}\\
			\leq & C (\delta^\gamma+\delta^{\gamma_0}).
		\end{split}
	\end{align*}
	
	Finally,  choosing $\gamma_0<1-1/q$ and $\delta>0$ sufficiently small. The proof is completed.
\end{proof}

\backmatter

\bmhead{Acknowledgements}
Y. Xu was partially supported by National Natural Science Foundation (NSF) of China under Grant No. U2441204.
B. Pei were partially sup-ported by NSFC under Grant No.12672040, Natural Science Basic Research
Program of Shaanxi under Grant No. 2026JC-JCQN-010, Fundamental Research Funds for the Central Universities.

\section*{Declarations}

{\bf Competing interests} The authors declare that they have no known competing financial interests or personal
 relationships that could have appeared to influence the work reported in this paper.


\begin{thebibliography}{10}
\bibitem{amann2011ordinary}
{ H.~Amann}, {\em Ordinary Differential Equations: An Introduction to
  Nonlinear Analysis}, vol.~13, Walter de Gruyter, 2011.

\bibitem{MR1385460}
{ H.~Bauer}, {\em Probability theory}, vol.~23 of De Gruyter Studies in
  Mathematics, Walter de Gruyter \& Co., Berlin, 1996.
\newblock Translated from the fourth (1991) German edition by Robert B. Burckel
  and revised by the author.

\bibitem{cao2023wong}
{ Q.~Cao, H.~Gao, and B.~Schmalfuss}, {\em Wong-{Z}akai type approximations
  of rough random dynamical systems by smooth noise}, Journal of Differential
  Equations, 358 (2023), pp.~218--255.

\bibitem{caraballo2004exponentially}
{ T.~Caraballo, P.~Kloeden, and B.~Schmalfu{\ss}}, {\em Exponentially stable
  stationary solutions for stochastic evolution equations and their
  perturbation}, Applied Mathematics and Optimization, 50 (2004), pp.~183--207.

\bibitem{cass2013integrability}
{ T.~Cass, C.~Litterer, and T.~Lyons}, {\em {Integrability and tail
  estimates for Gaussian rough differential equations}}, The Annals of
  Probability, 41 (2013), pp.~3026 -- 3050.

\bibitem{cheridito2003fractional}
{ P.~Cheridito, H.~Kawaguchi, and M.~Maejima}, {\em Fractional
  {Ornstein-Uhlenbeck} processes}, Electronic Journal of Probability, 8 (2003),
  pp.~1--14.

\bibitem{chueshov2005averaging}
{ I.~Chueshov and B.~Schmalfu{\ss}}, {\em Averaging of attractors and
  inertial manifolds for parabolic {PDE} with random coefficients}, Advanced
  Nonlinear Studies, 5 (2005), pp.~461--492.

\bibitem{chueshov2020synchronization}
{ I.~Chueshov and B.~Schmalfu{\ss}}, {\em Synchronization in
  Infinite-Dimensional Deterministic and Stochastic Systems}, Springer, 2020.

\bibitem{coddington1956theory}
{ E.~A. Coddington, N.~Levinson, and T.~Teichmann}, {\em Theory of ordinary
  differential equations}, American Institute of Physics, 1956.

\bibitem{cong2018nonautonomous}
{ N.~D. Cong, L.~H. Duc, and P.~T. Hong}, {\em Nonautonomous {Y}oung
  differential equations revisited}, Journal of Dynamics and Differential
  Equations, 30 (2018), pp.~1921--1943.

\bibitem{duc2020controlled}
{ L.~H. Duc}, {\em Controlled differential equations as rough integrals},
  Pure and Applied Functional Analysis, 4 (2022), pp.~1245--1271.

\bibitem{duc2018exponential}
{ L.~H. Duc, M.~J. Garrido-Atienza, A.~Neuenkirch, and B.~Schmalfu{\ss}},
  {\em Exponential stability of stochastic evolution equations driven by small
  fractional brownian motion with hurst parameter in $(1/2, 1)$}, Journal of
  Differential Equations, 264 (2018), pp.~1119--1145.

\bibitem{freidlin2012random}
{ M.~Freidlin and A.~Wentzell}, {\em Random Perturbations of Dynamical
  Systems}, Springer, 2012.

\bibitem{friz2020course}
{ P.~K. Friz and M.~Hairer}, {\em A Course on Rough Paths}, Springer, 2020.

\bibitem{friz2010multidimensional}
{ P.~K. Friz and N.~B. Victoir}, {\em Multidimensional Stochastic Processes
  as Rough Paths: Theory and Applications}, vol.~120, Cambridge University
  Press, 2010.

\bibitem{gao2021rough}
{ H.~Gao, M.~Garrido, A.~Gu, K.~Lu, and B.~Schmalfuss}, {\em Rough path
  theory to approximate random dynamical systems}, SIAM Journal on Applied
  Dynamical Systems, 20 (2021), pp.~997--1021.

\bibitem{givon2007strong}
{ D.~Givon}, {\em Strong convergence rate for two-time-scale jump-diffusion
  stochastic differential systems}, Multiscale Modeling \& Simulation, 6
  (2007), pp.~577--594.

\bibitem{gubinelli2004controlling}
{ M.~Gubinelli}, {\em Controlling rough paths}, Journal of Functional
  Analysis, 216 (2004), pp.~86--140.

\bibitem{hairer2019averaging}
{ M.~Hairer and X.-M. Li}, {\em {Averaging dynamics driven by fractional
  {B}rownian motion}}, The Annals of Probability, 48 (2020), pp.~1826--1860.

\bibitem{hasselmann1976stochastic}
{ K.~Hasselmann}, {\em Stochastic climate models part {I}. theory}, Tellus,
  28 (1976), pp.~473--485.

\bibitem{hong2022strong}
{ W.~Hong, S.~Li, and W.~Liu}, {\em Strong convergence rates in averaging
  principle for slow-fast mckean-vlasov spdes}, Journal of Differential
  Equations, 316 (2022), pp.~94--135.

\bibitem{inahama2022averaging}
{ Y.~Inahama}, {\em Averaging principle for slow-fast systems of rough
  differential equations via controlled paths}, To appear in Tohoku
  Mathematical Journal, 44 pages. arXiv: 2210.01334, 2025+.


\bibitem{inahama2025averaging}
{ Y.~Inahama}, {\em Averaging principle for rough slow-fast systems of level 3}, 46 pages. arXiv:2504.03110, 2025.



\bibitem{khasminskii1968on}
{ R.~Khasminskii}, {\em {On an averaging principle for It\^{o} stochastic
  differential equations}}, Kibernetica, 4 (1968), pp.~260--279.

\bibitem{li2022slow}
{ X.-M. Li and J.~Sieber}, {\em Slow-fast systems with fractional
  environment and dynamics}, The Annals of Applied Probability, 32 (2022),
  pp.~3964--4003.

\bibitem{liu2020averaging}
{ W.~Liu, M.~R{\"o}ckner, X.~Sun, and Y.~Xie}, {\em Averaging principle for
  slow-fast stochastic differential equations with time dependent locally
  lipschitz coefficients}, Journal of Differential Equations, 268 (2020),
  pp.~2910--2948.

\bibitem{lyons1994differential}
{ T.~Lyons}, {\em Differential equations driven by rough signals {(I)}: An
  extension of an inequality of {LC Young}}, Mathematical Research Letters, 1
  (1994), pp.~451--464.

\bibitem{Nualart2002}
{ D.~Nualart and A.~Rascanu}, {\em Differential equations driven by
  fractional {B}rownian motion}, Collectanea Mathematica, 53 (2002),
  pp.~55--81.

\bibitem{pei2020averaging}
{ B.~Pei, Y.~Inahama, and Y.~Xu}, {\em Averaging principles for mixed
  fast-slow systems driven by fractional {B}rownian motion}, Kyoto
  Journal of Mathematics, 63 (2023), pp.~721--748.

\bibitem{pei2020pathwise}
{ B.~Pei, Y.~Inahama, and Y.~Xu},{\em Pathwise unique
  solutions and stochastic averaging for mixed stochastic partial differential
  equations driven by fractional {B}rownian motion and {B}rownian motion}, To
  appear in Stochastic Analysis and Applications, 35 pages, arXiv:2004.05305.

\bibitem{pei2021averaging}
{ B.~Pei, Y.~Inahama, and Y.~Xu},{\em Averaging principle
  for fast-slow system driven by mixed fractional {B}rownian rough path},
  Journal of Differential Equations, 301 (2021), pp.~202--235.

\bibitem{pei2023almost}
{ B.~Pei, B.~Schmalfuss, and Y.~Xu}, {\em Almost sure averaging for
  evolution equations driven by fractional {B}rownian motions}, SIAM Journal on Applied Dynamical Systems, 23(4)(2024), pp. ~2807-2852.




\bibitem{riedel2017rough}
{ S.~Riedel and M.~Scheutzow}, {\em Rough differential equations with
  unbounded drift term}, Journal of Differential Equations, 262 (2017),
  pp.~283--312.


\bibitem{rockner2021strong}
{ M.~R{\"o}ckner, X.~Sun, and Y.~Xie}, {\em Strong convergence order for
  slow--fast {McKean--Vlasov }stochastic {D}ifferential equations}, Annales de
  l???Institut Henri Poincar{\'e}-Probabilit{\'e}s et Statistiques, 57 (2021),
  pp.~547--576.

\bibitem{schmalfuss1998random}
{ B.~Schmalfuss}, {\em A random fixed point theorem and the random graph
  transformation}, Journal of Mathematical Analysis and Applications, 225
  (1998), pp.~91--113.

\bibitem{shen2022averaging}
{ G.~Shen, J.~Xiang, and J.-L. Wu}, {\em Averaging principle for
  distribution dependent stochastic differential equations driven by fractional
  {B}rownian motion and standard {B}rownian motion}, Journal of Differential
  Equations, 321 (2022), pp.~381--414.

\bibitem{sussmann1978gap}
{ H.~J. Sussmann}, {\em On the gap between deterministic and stochastic
  ordinary differential equations}, The Annals of Probability, 6 (1978),
  pp.~19--41.

\bibitem{wu2022fast}
{ F.~Wu and G.~Yin}, {\em Fast-slow-coupled stochastic functional
  differential equations}, Journal of Differential Equations, 323 (2022),
  pp.~1--37.

\bibitem{xu2011averaging}
{ Y.~Xu, J.~Duan, and W.~Xu}, {\em An averaging principle for stochastic
  dynamical systems with {L}{\'e}vy noise}, Physica D: Nonlinear Phenomena, 240
  (2011), pp.~1395--1401.

\bibitem{xu2015stochastic}
{ Y.~Xu, B.~Pei, and R.~Guo}, {\em Stochastic averaging for slow-fast
  dynamical systems with fractional {B}rownian motion}, Discrete and Continuous
  Dynamical Systems-B, 20 (2015), pp.~2257--2267.

\bibitem{xu2017stochastic}
{ Y.~Xu, B.~Pei, and J.-L. Wu}, {\em Stochastic averaging principle for
  differential equations with non-lipschitz coefficients driven by fractional
  {B}rownian motion}, Stochastics and Dynamics, 17 (2017), p.~1750013.
\end{thebibliography}
\end{document}